\documentclass[reqno]{amsart}
\usepackage{amsfonts}
\usepackage{amsmath,amssymb,amsthm,bm,bbm}
\usepackage{amscd}
\usepackage{color}
\usepackage{caption}
\usepackage{float}
\usepackage{subcaption}
\usepackage{graphicx}
\usepackage{geometry}
\usepackage{mathrsfs}
\usepackage{enumitem}
\usepackage{makecell}
\usepackage{hyperref}

\usepackage{etoolbox}
\patchcmd{\section}{\scshape}{\bfseries}{}{}
\makeatletter
\renewcommand{\@secnumfont}{\bfseries}
\makeatother

\newcommand{\B}{{\mathcal B}}

\newcommand{\Q}{{\mathbb Q}}

\newcommand{\cW}{{\mathcal {W}}}
\newcommand{\cF}{{\mathcal {F}}}

\newcommand{\cP}{{\mathcal {P}}}
\newcommand{\N}{{\mathbb N}}

\newcommand{\QQ}{{\mathbb{Q}}}
\newcommand{\RR}{{\mathbb{R}}}

\newcommand{\card}{\text{card}}

\newcommand{\Area}{\text{Area}}

\newcommand{\eps}{\varepsilon}
\newcommand\blue[1]{\textcolor{blue}{#1}}
\numberwithin{equation}{section}

\newtheorem{thm}{Theorem}[section]
\newtheorem{lem}[thm]{Lemma}
\newtheorem{prop}[thm]{Proposition}
\newtheorem{cor}[thm]{Corollary}

\theoremstyle{definition}

\newtheorem{defn}[thm]{Definition}

\newtheorem{remark}[thm]{Remark}

\begin{document}
	\title{Delone sets associated with badly approximable triangles}
	\author{Shigeki Akiyama} \address{
		Institute of Mathematics, University of Tsukuba, 1-1-1 Tennodai, Tsukuba, Ibaraki, 305-8571 Japan
	} \email{akiyama@math.tsukuba.ac.jp}
	\author{Emily R. Korfanty} \address{Department of Mathematical and Statistical Sciences, University of Alberta, Edmonton, AB, T6G 2G1, Canada}\email{ekorfant@ualberta.ca}
	\author{Yan-li Xu$^*$} \address{Department of Mathematics and Statistics, Central China Normal University, Wuhan, 430079, China} \email{xu\_yl@mails.ccnu.edu.cn}
	\date{\today}
	
	\thanks{* Corresponding author.}
	
	\begin{abstract}
		We construct new Delone sets associated with badly approximable numbers which are expected to have rotationally invariant diffraction. We optimize the discrepancy of corresponding tile orientations by investigating the linear equation $x+y+z=1$ where $\pi x$, $\pi y$, $\pi z$ are three angles of a triangle used in the construction and $x$, $y$, $z$ are badly approximable. In particular, we show that there are exactly two solutions that
		have the smallest partial quotients by lexicographical ordering.
	\end{abstract}
	\maketitle
	
	\section{Introduction}
	
	The study of non-periodic structures and their diffraction has been a topic of great interest since the discovery of quasicrystals in 1984 
	by Dan Shechtman 
	\cite{Shechtman-et-al:84}.  The diffraction from these materials exhibit sharp patterns of bright spots, known as Bragg peaks, despite having a non-periodic atomic structure. This raised a compelling question: \emph{Which non-periodic structures exhibit sharp diffraction patterns?} 
	
	Today, much is known about non-periodic structures when 
	the local patterns are finite up to translations; this property is known as finite local complexity.
	We refer the readers to \cite{Baake-Gahler:16, Baake-Grimm:13} for a broad range of examples and their corresponding theory of pure point diffraction.  
	However, diffraction is less understood for structures that do not have finite local complexity, especially for substitution tilings with statistical circular symmetry. Here, statistical circular symmetry refers to the orientations of the tiles being uniformly distributed on the unit circle
	when ordered according to the self-similar structure (see~\cite{Frettloh:08} for a definition). 
	The paradigm of such structures is the pinwheel tiling \cite{Radin:94}. Of the known tilings with statistical circular symmetry (see \cite{Frettloh:08,Frettloh-Harriss-Gahler,Sadun:98} for examples), the pinwheel tiling has been most thoroughly studied \cite{Baake-Frettloh-Grimm:07, Baake-Frettloh-Grimm:07b, Grimm-Deng:2011, MPS:06, Postnikoff:2004}.  Despite this, little is known about the pinwheel diffraction, except that it is rotationally invariant with a Bragg peak of unit intensity at the origin.
	
	The pinwheel tiling is a non-periodic tiling of $\RR^2$ by a right triangle with side lengths 1, 2, and $\sqrt{5}$. It is an inflation tiling constructed via the subdivision rule shown in Figure~\ref{fig:pinwheel-sub}.  More specifically, starting from an initial triangle, one iteratively applies an inflation by $\sqrt{5}$ and subdivides each tile into $5$ smaller, congruent triangles according to the subdivision rule.  For the pinwheel tiling, there is a canonical choice of a distinguished point within each tile, and together these points form the usual Delone set associated with the pinwheel tiling.
	A patch of the pinwheel tiling and its Delone set is shown in Figure~\ref{fig:pinwheel-patch}.
	\begin{figure}[ht]
		\begin{center}     \includegraphics{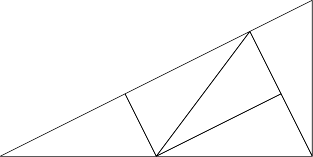}
		\end{center}
		\caption{The pinwheel subdivision rule.}
		\label{fig:pinwheel-sub}
	\end{figure}
	\begin{figure}[ht]
		\begin{center}     
			\includegraphics[width=8cm]{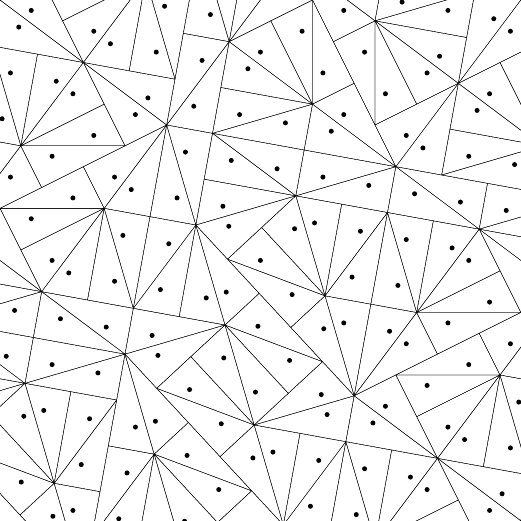}
		\end{center}
		\caption{The pinwheel tiling and its associated Delone set.}
		\label{fig:pinwheel-patch}
	\end{figure}
	
	The statistical circular symmetry of the pinwheel tiling is due to the key angle~$\arctan(\frac{1}{2})$, which is incommensurate with $\pi$.
	More generally, for primitive substitution tilings in $\RR^2$, statistical circular symmetry is equivalent to existence of a level-$n$ ($n\geq 1$) supertile containing two copies
	of the same prototile differing in orientation by an angle $\alpha \notin \pi \QQ$ (see \cite[Proposition~3.4 and Theorem~6.1]{Frettloh:08}). The essential reason for this fact is 
	that the map $x\to x+ \alpha$ 
	specifies an irrational rotation on the torus $S^1$, and by a theorem of Weyl \cite{Weyl:16}, the orbit of an irrational rotation is uniformly distributed on $S^1$. 
	
	In this paper, we are interested in the rate of convergence of the distribution of angles 
	to the uniform distribution, i.e., the discrepancy. It is
	well-known that $x\to x+ \alpha \pmod{1}$ attains the smallest possible discrepancy up to constant factors 
	when $\alpha$ is badly-approximable, 
	i.e., when its partial quotients are bounded. Moreover, if this bound is small, then the above constant also becomes small (see ~\cite[Chapter~2,~Theorem~3.4]{Kuipers-Niederreiter:74}). Badly approximable angles often appear in phyllotaxis. One such example is the golden angle $\pi \omega$
	where
	$$
	\omega=\frac{\sqrt{5}-1}{2}=
	\cfrac{1}{1+\cfrac{1}{1+\cfrac{1}{\ddots}}}
	=[1,1,\dots] \,.
	$$
	The partial quotients of $\omega$ are minimal, and therefore, the irrational rotation by $\pi\omega$ leads to the fastest convergence to the uniform distribution.
	
	In this regard, pinwheel tiling is not ideal. There are currently no known bounds for the partial quotients of 
	$$
	\frac{\arctan(1/2)}{\pi}=[6, 1, 3, 2, 5, 1, 6, 5,\dots].
	$$
	Due to the Gelfond-Schneider Theorem, it is known that $\arctan(1/2)/\pi$ is transcendental.  In particular, this implies that its expansion is not eventually periodic. Though these first several terms are fairly small, one can find large partial quotients $583, 1990, 116880, 213246\dots$ in its expansion at positions $53, 1171, 4806, 109153, \dots$. 
	Since the set of badly approximable numbers has measure zero (see, for example, \cite[Chapter 11, Theorem 196]{HW} or    \cite[Chapter 2, Theorem 29]{Khinchin:97}),
	it is natural to guess that $\arctan(1/2)/\pi$ is \emph{not} badly approximable.
	Further, by ergodicity of the continued fraction map, almost all numbers are normal with respect to the Gauss measure \cite{Khinchin:97,KN:00}, and consequently are not badly approximable.
	Note also that the right angle $\pi/2$ that appears in the pinwheel tiling
	is the antipode of the badly approximable angles. 
	
	Similar to the pinwheel tiling, the key angles for the other aforementioned tilings with statistical circular symmetry are also not likely to be badly approximable.
	Motivated by this, we construct new tilings and associated Delone sets by triangles where every angle is the product of $\pi$ and a badly approximable number.  We start from the subdivision rule shown in Figure~\ref{fig:subdivision-rule-new}.  
	\begin{figure}[ht]
		\centering
		\includegraphics[width=9 cm]{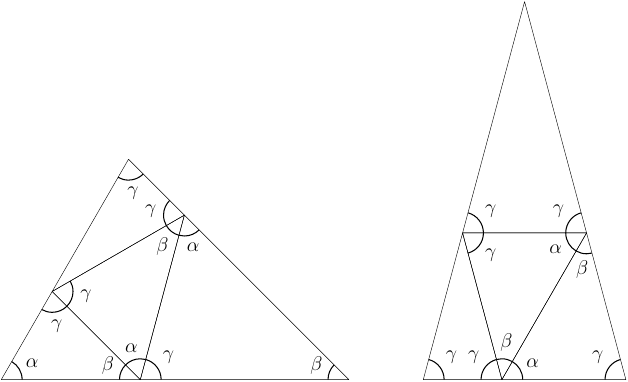}
		\caption{Subdivision rule for triangles with angles $\alpha$, $\beta$, $\gamma$.  The triangle on the left is scalene, and the triangle on the right is isosceles.  This rule is valid for any solutions of~$\alpha+\beta+\gamma=\pi$.}
		\label{fig:subdivision-rule-new}
\end{figure}
This subdivision rule has the special property that the angles~$\alpha,\beta,\gamma$ can be chosen to be \emph{any} angles satisfying $\alpha + \beta + \gamma = \pi$.

In particular, if one can choose $\alpha,\beta,\gamma$
so that~$\alpha/\pi, \beta/\pi$ and $\gamma/\pi$ 
are badly approximable numbers, then the remaining angle $\pi - 2\gamma$ is 
also a badly approximable multiples of $\pi$.

This leads us to our target equation
$$
x+y+z=1 \,,
$$
where $x, y, z$ are badly approximable numbers and $\alpha = \pi x, \beta = \pi y, \gamma = \pi z$ are the angles of the corresponding triangle.  We are especially interested in solutions such that the partial quotients of $x, y, z$ are small by lexicographical ordering. In this case, we refer to the triangle with angles $\pi x, \pi y, \pi z$ as an \emph{optimal badly approximable triangle}.



It is easy to see that if each term in the continued fraction expansion of 
$x,y,z$ does not exceed two, the equation $x+y+z=1$ has no solution. 
Therefore, we seek a solution $x,y,z$ such that, for each of these numbers, the first partial quotient does not exceed three, and the remaining quotients are no greater than two.  To our surprise, we can show that the equation $x+y+z=1\ (x\le y\le z)$ has exactly two solutions under this restriction:
$$
x=2-\sqrt{3}=[3,1,2,1,2,1,2\ldots],\ y=z=\frac{\sqrt{3}-1}2=[2,1,2,1,2,1,\ldots]\,,
$$
and
$$
x=y=\frac{2-\sqrt{2}}2=[3,2,2,2,2,2,\ldots],\ z=\sqrt{2}-1=[2,2,2,2,2,\ldots]\, ;
$$
see Theorem~\ref{Main}.
The proof of this fact requires careful case analysis on infinitely many sub-cases. 
Based on this main result, we can then easily conclude that the equation $x+y=z\ (x\le y)$ has exactly four solutions under the same conditions; see Theorem~\ref{Main2}. Furthermore, our method gives uncountably many explicit solutions when the partial quotients of $x,y,z$ do not exceed three; see Theorem~\ref{Main3}.  

Combining these results on badly approximable numbers with the subdivision rule of Figure~\ref{fig:subdivision-rule-new}, we obtain Delone sets associated with tilings that have optimal statistical circular symmetry.  More specifically, the Delone sets are produced from  optimal badly approximable triangles, so that the discrepancy is minimized.


To construct our Delone sets, we largely follow the threshold method for multiscale substitution schemes considered in \cite{Smi-Solo:21}, but we use contractions
described by a graph directed iterated function system to give a concise presentation. The main idea is to subdivide the triangles until the areas reach a given threshold, and then renormalize them to obtain larger and larger patches.  By choosing a suitable point within each triangle (e.g. the centroids), we get a sequence of finite point sets.  We prove the existence of a Delone limit set for this sequence in the \emph{Chabauty--Fell topology} \cite{Chabauty:50,Fell:62} (see Theorem~\ref{thm:convergence}).  A patch of a Delone set obtained from the subdivision rule in Figure~\ref{fig:subdivision-rule-new} for using optimal badly approximable triangles is shown in Figure~\ref{fig:optimal1-patch}.  
\begin{figure}[ht]
	\begin{center}     
		\includegraphics[width=8cm]{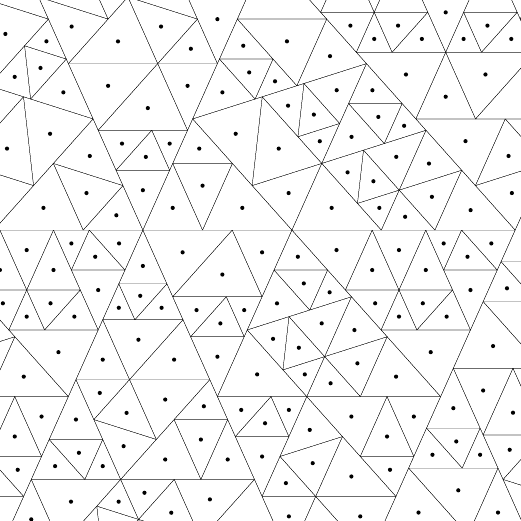}
	\end{center}
	\caption{A new tiling by optimal badly approximable triangles and its associated Delone set, constructed via the subdivision rule shown in Figure~\ref{fig:subdivision-rule-new} with $\alpha = (2-\sqrt{3})\pi$ and~${\beta=\gamma=\frac{(\sqrt{3}-1)\pi
			}{2}}$.   }
	\label{fig:optimal1-patch}
\end{figure}
	
	The paper is organized as follows. In Section~\ref{sec:main-results-1}, we provide the required background and definitions, and state our main results on badly approximable numbers.  In Section~\ref{sec:main-results-2}, we describe our construction of Delone sets using graph directed iterated function systems. In Section~\ref{sec:specific}, we return to the original motivation and discuss the Delone sets obtained from the subdivision rule shown in Figure~\ref{fig:subdivision-rule-new} for the optimal badly approximable triangles associated with Theorem~\ref{Main}. Then, in  Section~\ref{sec:proof_main123}, we prove Theorem~\ref{Main}, Theorem \ref{Main2} and Theorem~\ref{Main3}. Finally, in Section~\ref{sec:open}, we give several open problems.

	\section{Solving \texorpdfstring{$x+y+z=1$}{x+y+z=1} in  badly approximable numbers}\label{sec:main-results-1}
	
	In this section, we will state our main results on badly approximable numbers.
	Their proofs are found in Section \ref{sec:proof_main123}.
	Let us start some definitions.
	
	\begin{defn}
An irrational number $x \in (0,1)$ is called \emph{badly approximable} if the partial quotients in the continued fraction expansion
$$
x=[a_1(x),a_2(x),\dots]=\cfrac 1{a_1(x)+\cfrac 1{ a_2(x)+ \cfrac 1{\ddots}}}\,, \quad a_j(x) \in \mathbb{Z}_+\,, \ j=1,2,\ldots \,,
$$
are bounded, i.e.\ if $\sup_{k \geq 1}a_k(x)<\infty$.
\end{defn}
Equivalently, a number $x\in (0,1)$ is badly approximable if and only if there exists some $\varepsilon>0$ with the property that \begin{equation*}
\left|x-\frac{p}{q}\right|\geq \frac{\varepsilon}{q^2} \,,
\end{equation*}
for all rational numbers $\frac{p}{q}$; see \cite[Chapter 11]{HW} or \cite[Theorem 23]{Khinchin:97}.
For $x=[a_1(x),a_2(x),\dots]\in (0,1)$, by using the Gauss map
$$
T(x)=\frac 1x -\left\lfloor \frac 1x \right\rfloor\,,
$$
we have
$$
T^{k-1}(x)=[a_{k}(x),a_{k+1}(x),a_{k+2}(x),\dots] \,,
$$
and $a_k(x)=\lfloor 1/T^{k-1}(x) \rfloor$ for all $k\geq 1$.

\begin{defn}
A continued fraction $x = [a_1,a_2,\dots]\,$ is \textit{eventually periodic} if there are integers $N\geq 0$ and $k\geq 1$ with $a_{n+k}=a_n$ for all $n \geq N$.
Such a continued fraction will be written
\[
x = [a_1,\dots,a_{N-1},\overline{a_N,\dots,a_{N+k-1}}] \,.
\]
\end{defn}
We use the notation $(a_N,\dots,a_{N+k-1})^\ell$ to denote the repetition of the numbers $a_N,\dots,a_{N+k-1}$ in the continued fraction $\ell\geq 0$ many times. We write $(a_j)^\ell$ for the repetition of a single number $a_j$.
For convenience, in the case where $x\in(0,1)\cap\QQ$ we use the notation
\[
x = [a_1,a_2,\dots,a_n,\infty] =\frac{1}{a_1+\frac{1}{a_2+\frac{1}{\ddots + \frac{1}{a_n}}}}\,.
\]
\begin{defn}
Define the \textit{cylinder set} of $b_1,\dots,b_n\in\mathbb{N}$ by
\[
I(b_1,\dots,b_n)= \{x\in(0,1) \,:\, x=[x_1,x_2,\dots]\,, x_i=b_i\ for\ 1 \leq i\leq n\}\,.
\]
\end{defn}
The set $I(b_1,\dots , b_n)$ is an interval with endpoints
\[
\frac{P_n+P_{n\blue{-}1}}{Q_n+Q_{n\blue{-}1}}\quad and\quad \frac{P_n}{Q_n} \,,
\]
for $n\geq 1$, where
$$
P_n=b_nP_{n-1}+P_{n-2}\,,\quad Q_n=b_nQ_{n-1}+Q_{n-2} \,,
$$
with
\[
\begin{pmatrix}
P_{-1} & P_0\\
Q_{-1} & Q_0
\end{pmatrix}=
\begin{pmatrix}
1 & 0\\
0 & 1
\end{pmatrix}\,.
\]

Let us define our linear problem for badly approximable numbers
more precisely.
An irrational number $x\in (0,1)$ is $B$-bad if $a_k(x)\le B$ holds
for all $k \geq 1$. Let $\B_B$ be the set of all $B$-bad numbers in $(0,1)\backslash \QQ$. For $j\ge 0$, we define the set
$$
\B_{B,j}= \B_{B+1} \cap T^{-j}(\B_B) \,,
$$
i.e., $\B_{B,j}$ is the set of irrational
numbers which satisfy
\begin{equation*}
\begin{cases}
	a_k\le B+1 & k \leq j\\
	a_k\le B & k > j \,.
\end{cases}
\end{equation*}
Clearly, we have
$$\B_B=\B_{B,0}\subset \B_{B,1} \subset \B_{B,2} \subset \cdots\,.$$
Further, we define $\B^*_B=\bigcup_{j=0}^{\infty} \B_{B,j}$ to be the set of eventually $B$-bad numbers in $\B_{B+1}$. In this paper, we are interested in the additive structure of $\B_{B,j}$ and $\B^*_B$. We begin with a simple lemma.

\begin{lem}
\label{Triv}
\emph{
	For $x=[a_1,a_2,a_3,\dots]\in (0,1)$, we have $$
	1-x=\begin{cases} [1,a_1-1,a_2,a_3,\dots] & a_1\ge 2\\
		[1+a_2,a_3,\dots] & a_1=1\,.\end{cases}
	$$
}
\end{lem}

\begin{proof}
Putting $x=1/(a_1+y)$ with $y\in (0,1)$, we see that
$$
1-x=\cfrac {1}{1+\frac 1{a_1-1+y}} \,,
$$
from which the result easily follows.
\end{proof}

\begin{cor}\label{cor:Trivial}
\emph{
	An irrational number $x$ is in $\B_{2,1}$ if and only if $1-x$ is also in $\B_{2,1}$.
}
\end{cor}

\begin{remark}
The property of $\B_{2,1}$ described in Corollary~\ref{cor:Trivial} does not hold in $\B_2$ or in $\B_{2,j}$ for any~$j\geq 2$.
\end{remark}

\begin{remark}\label{rem:no-B2-solution}
Lemma~\ref{Triv} shows that the equation
$
x+y=1\ (x,y\in \B_{2},\ x\le y)
$
is trivially solved and has the set of solutions
\[
\{ (x,1-x) \ |\ x\in \B_{2}\cap [0,1/2) \} \,.
\]
In particular, the equation has uncountably many different solutions.  However, our equation of interest $x+y+z=1$ has no solutions in $\B_2$.  Indeed, if $x,y,z\in \B_2$, then we also have $x,y,z \in I(1) \cup I(2) = [\frac{1}{3},1)$. However, if we also have $x+y+z=1$, then the only possible solution is $x=y=z=\frac{1}{3}\in\mathbb{Q}$, which contradicts irrationality of $x,y,z\in\B_2$.
\end{remark}

Our main results are as follows:

\begin{thm}\label{Main}
\emph{
	The equality
	$
	x+y+z=1\ (x,y,z\in \B_{2,1},\ x\le y\le z)
	$
	has exactly two solutions
	$$
	x=2-\sqrt{3}=[3,\overline{1,2}],\ y=z=\frac{\sqrt{3}-1}2=[\overline{2,1}]\,,
	$$
	and
	$$
	x=y=\frac{2-\sqrt{2}}2=[3,\overline{2}],\ z=\sqrt{2}-1=[\overline{2}]\,.
	$$
}
\end{thm}

By using Lemma \ref{Triv}, we may rephrase Theorem \ref{Main} as follows:

\begin{thm}
\label{Main2}
\emph{
	The equality
	$
	x+y=z\ (x,y,z\in \B_{2,1}\,, x \leq y)
	$
	has exactly four solutions
	$$
	x=2-\sqrt{3}=[3,\overline{1,2}],\ y=\frac{\sqrt{3}-1}2=[\overline{2,1}],\ z=\frac{3-\sqrt{3}}{2}=[1,1,1,\overline{2,1}]\,,
	$$
	$$
	x=y=\frac{\sqrt{3}-1}2=[\overline{2,1}],\ z=\sqrt{3}-1=[\overline{1,2}]\,,
	$$
	$$
	x=y=\frac{2-\sqrt{2}}2=[3,\overline{2}],\ z=2-\sqrt{2}=[1,1,\overline{2}]\,,
	$$
	and
	$$
	x=\frac{2-\sqrt{2}}2=[3,\overline{2}],\ y=\sqrt{2}-1=[\overline{2}],\ z=\frac{\sqrt{2}}{2}=[1,\overline{2}]\,.
	$$
}
\end{thm}

In 1954, Hall \cite{Hall:47}
proved that $\B_4+\B_4$ contains an interval. 
Subsequently, Freiman~\cite{Freiman:73} and Schecker~\cite{Schecker:77} showed that $\B_3+\B_3$ contains an interval.
This implies that the equalities
$x+y+z=1\ (x,y,z\in \B_{3},\ x\le y\le z)$
and
$x+y=z\ (x,y,z\in \B_{3},\ x\le y)$
both have uncountably many solutions.

These results were proven using a criterion for Cantor sets which, when satisfied, implies that the arithmetic sum of two Cantor sets contains an interval.
By this method,
it may be challenging to make the solutions explicit, i.e.\ we know that one can choose
any $z\in \B_3$ in the interval,
but it is not clear which numbers $x,y\in \B_3$ satisfy the equation.
In contrast, our method gives explicit solutions in $\B^*_{2}$ and~$\B_3$:

\begin{thm}
\label{Main3}
\emph{
	The equality
	$
	x+y+z=1\ (x,y,z\in \B^*_{2},\ x\le y\le z)
	$
	has infinitely many solutions.
	Furthermore, uncountably many solutions of the equality
	$
	x+y+z=1\ (x,y,z\in \B_{3},\ x\le y\le z)
	$
	can be constructed explicitly.
}
\end{thm}

\section{Construction of Delone Sets}\label{sec:main-results-2}

In this section, we construct Delone sets starting from an arbitrary subdivision rule described by a graph directed iterated function system satisfying certain basic requirements, all of which are satisfied by our subdivision rule shown in Figure~\ref{fig:subdivision-rule-new}.
Let us begin with some definitions.

\subsection{Delone sets and the Chabauty--Fell topology}
We restrict our attention to $\RR^2$.  Throughout, we denote the Euclidean norm on $\RR^2$ by $\|\cdot\|$, and we write $B(x,r)$ to denote the open ball of radius $r$ centered at $x$, i.e.\ $B(x,r)=\{y\in\RR^2\,:\, \|y-x\|< r\}$. Similarly, we write $\overline{B(x,r)}$ for the closed ball, i.e.\ $\overline{B(x,r)}=\{y\in\RR^2\,:\, \|y-x\| \leq r\}$. We begin by recalling the definition of a Delone set.

\begin{defn}
Let $\Lambda$ be a closed subset of $X\subseteq \RR^2$.
\begin{enumerate}[label=(\roman*)]
	\item If there exists an $r>0$ such that for any $x \in X$, $\card (B(x,r) \cap \Lambda) \leq 1$, then $\Lambda$ is \emph{$r$-uniformly discrete} in $X$.
	\item If there exists an $R>0$ such that for any $x \in X$, $\overline{B(x,R)} \cap \Lambda\neq\emptyset$, then $\Lambda$ is \emph{$R$-relatively dense} in $X$.
	\item We say that $\Lambda$ is a \emph{$(r,R)$-Delone set} in $X$ if $\Lambda$ is both $r$-uniformly discrete in $X$ and $R$-relatively dense in $X$.
\end{enumerate}
\end{defn}
\begin{defn}
Let $X\subseteq \RR^2$ be closed.  We define the following spaces of point sets:
\begin{enumerate}[label=(\roman*)]
	\item Denote by $2^X$ the set of all closed subsets of $X$;
	\item Denote by $\cW_{r,R}(X)$ the set of closed subsets $\Lambda \subseteq \RR^2$ such that $\Lambda$ is an $(r,R)$-Delone set in $X$.
\end{enumerate}
\end{defn}
\begin{remark}
If $X$ and $Y$ are closed subsets of $\RR^2$ and $Y \subseteq X$, then any set which is $r$-uniformly discrete and $R$-relatively dense in $X$ must also be $r$-uniformly discrete and $R$-relatively dense in $Y$.  In particular, we have $\cW_{r,R}(X)\subseteq \cW_{r,R}(Y)$.
\end{remark}

Next, we will construct a sequence of finite point sets and show the existence of a subsequence converging to a Delone set in the \textit{Chabauty--Fell topology} \cite{Chabauty:50,Fell:62}.  This topology is also commonly referred to as the \textit{Chabauty topology} or the \textit{Fell topology}, which are defined more generally in any topological space. In the case of locally compact groups, it is also called the \textit{local rubber topology} \cite{Baake-Lenz:04}. Since we work in $\RR^2$, the Chabauty--Fell topology is metrizable and induced by the following metric (see {\cite[Appendix~A]{Smi-Solo:22}}):
\begin{defn}[Chabauty--Fell Topology]
For each $\Lambda_1,\Lambda_2\in2^{X}$, define
\[
d(\Lambda_1,\Lambda_2) = \inf\ \{\{1\}\cup\{\varepsilon>0\,:\,\Lambda_1\cap B(0,\textstyle{\frac{1}{\varepsilon}}) \subseteq \Lambda_2 + B(0,\varepsilon)\ \text{and}\ \Lambda_2\cap B(0,\textstyle{\frac{1}{\varepsilon}}) \subseteq \Lambda_1 + B(0,\varepsilon)\}\}\,.
\]
The map $d:2^X\times2^X \rightarrow [0,\infty)$ is a metric on $2^X$ inducing the Chabauty--Fell topology.
\end{defn}

\begin{remark}
It has long been known that the space $2^X$ is compact in the Fell topology \cite{Fell:62}.  Furthermore, it is easy to show that the space $\cW_{r,R}(X)$ is also compact in the Fell topology.  As we will use this fact, we give a proof for $X = \RR^2$ in Appendix~\ref{sec:compactness}.
\end{remark}

We will deduce the existence of a subsequence converging to a Delone set using the compactness of $\mathcal{W}_{r,R}(\RR^2)$; however, our sequence will not be contained in $\mathcal{W}_{r,R}(\RR^2)$ because no finite set is relatively dense in $\RR^2$. Thus, we will use the following easy observation about the Chabauty--Fell topology:
\begin{prop}\label{prop:Fell-compact-convergence}
\emph{
	Let $\{\Lambda_n\}_{n \geq 1}$ be a sequence in $2^X$ and $\Lambda \in 2^X$. Let $\{R_n\}_{n \geq 1}$ be a sequence of positive real numbers with $\lim_{n\rightarrow\infty} R_n = \infty$. Then
	$\Lambda_n$ converges to $\Lambda$ if and only if $\Lambda_n \cap B(0,R_n)$ converges to $\Lambda$ in the Chabauty--Fell topology.
}
\end{prop}

\begin{proof}
First, assume that $\Lambda_n$ converges to $\Lambda$ in the Chabauty--Fell topology. Then, for any $\varepsilon>0$, there exists a constant $N_1$ such that for any $n \geq N_1$, we have $d(\Lambda_n,\Lambda)<\varepsilon$. That is,
\begin{equation}\label{def_1}
	\Lambda_n \cap B(0,\tfrac{1}{\varepsilon}) \subseteq \Lambda+B(0,\varepsilon) \,,
\end{equation}
and
\begin{equation}\label{def_2}
	\Lambda \cap B(0,\tfrac{1}{\varepsilon}) \subseteq \Lambda_n+B(0,\varepsilon) \,.
\end{equation}
For the above $\varepsilon$, there must exist a constant $N_2$ such that for any $n \geq N_2$, $R_n>\tfrac{1}{\varepsilon} +\varepsilon$. Let $M=\max\{N_1,N_2\}$. Then by \eqref{def_1}, we have
\begin{equation}\label{def_3}
	\Lambda_n \cap B(0,R_n) \cap B(0,\tfrac{1}{\varepsilon}) = \Lambda_n \cap B(0,\tfrac{1}{\varepsilon}) \subseteq \Lambda+B(0,\varepsilon) \quad \forall n \geq M\,.
\end{equation}
Furthermore, we get
\begin{equation}\label{def_4}
	\Lambda \cap B(0,\tfrac{1}{\varepsilon}) \subseteq \Lambda_n \cap B(0,R_n)+B(0,\varepsilon)\,.
\end{equation}
Next, let $x\in \Lambda \cap B(0, \tfrac{1}{\varepsilon})$. By \eqref{def_2}, we have $x \in \Lambda_n + B(0,\varepsilon)$. Thus, there exists some $y\in \Lambda_n$ and some $z\in B(0,\varepsilon)$ such that $x=y+z$.  Since $R_n > \tfrac{1}{\varepsilon} + \varepsilon$ and $x\in B(0, \tfrac{1}{\varepsilon})$ we have
\[
\|y\| \leq \|x\| + \|x-y\| = \|x\| + \|z\| < \frac{1}{\varepsilon} + \varepsilon < R_n\,,
\]
so $y\in \Lambda_n \cap B(0,R_n)$.  In particular, we have $x = y + z \in \Lambda_n \cap B(0,R_n) + B(0,\varepsilon$).

By \eqref{def_3} and \eqref{def_4}, we have shown that $d(\Lambda_n \cap B(0,R_n),\Lambda)<\varepsilon$ for all $n \geq M$, so  $\Lambda_n \cap B(0,R_n)$ converges to $\Lambda$ in the Chabauty--Fell topology.

Conversely, suppose that $\Lambda_n \cap B(0,R_n)$ converges to $\Lambda$ in the Chabauty--Fell topology. For any $\varepsilon>0$, there exists a constant $N_1$ such that for any $n \geq N_1$, we have $d(\Lambda_n\cap B(0,R_n),\Lambda)<\varepsilon$. For the above $\varepsilon$, there must exist a constant $N_2$ such that $R_n>\tfrac{1}{\varepsilon}$ for any $n \geq N_2$. Let $M=\max\{N_1,N_2\}$.  Then, for all $n \geq M$, we get
\[
\Lambda_n\cap B(0,\tfrac{1}{\varepsilon})=\Lambda_n\cap B(0,R_n) \cap B(0,\tfrac{1}{\varepsilon}) \subseteq \Lambda+B(0,\varepsilon) \,,
\]
and
\[
\Lambda \cap B(0,\tfrac{1}{\varepsilon}) \subseteq \Lambda_n\cap B(0,R_n) + B(0,\varepsilon) \subseteq \Lambda_n+B(0,\varepsilon) \,,
\]
so $d(\Lambda_n, \Lambda)< \eps$. Therefore, $\Lambda_n $ converges to $\Lambda$ in the Chabauty--Fell topology.
\end{proof}
In other words, the convergence of a sequence in the Chabauty--Fell topology is equivalent to the convergence of its restriction to larger and larger patches. Provided that we can extend our finite point sets to elements of $\mathcal{W}_{r,R}(\RR^2)$, we can use this fact and the compactness of $\mathcal{W}_{r,R}(\RR^2)$ to prove the existence of a Delone limit set.

\subsection{Threshold method}\label{sec:threshold}
To any subdivision rule on a finite set of tiles in $\RR^2$, there is an associated graph-directed iterated function system (GIFS) consisting of similitudes.  In this section, starting from a GIFS of this type, we construct a corresponding Delone set.  In particular, starting from some initial tile, we apply the GIFS until the area of every tile is below some threshold $\varepsilon > 0$. Once there are no tiles of area greater than $\varepsilon$, we inflate the finite patch by $\frac{1}{\sqrt{\varepsilon}}$.  We call this an \textit{$\varepsilon$-rule}, which is defined more precisely below.

\subsubsection{\textbf{GIFS}}
Let $(V,\Gamma)$ be a directed graph with vertex set $V$ and directed-edge set $\Gamma$ where
both $V$ and $\Gamma$ are finite. Denote the set of edges from $j$ to $i$ by $\Gamma_{j,i}$ and assume that for any $j\in V$, there is at least one edge starting from vertex $j$. Furthermore, assume that for each edge $e\in\Gamma$, there is a corresponding contractive  similitude $g_e:\mathbb{R}^n\rightarrow\mathbb{R}^n$.
We call $(g_e)_{e\in \Gamma}$ a \emph{graph-directed IFS (GIFS)} (see \cite{MW88}). The invariant sets of this GIFS, also called \emph{graph-directed sets},
are the unique non-empty compact sets $(T_j)_{j\in V}$ satisfying
\begin{equation}\label{GIFS}
T_j=\bigcup\limits_{i\in V}\bigcup\limits_{e\in\Gamma_{j,i}}g_e(T_i),\quad  j\in V \,.
\end{equation}

Let $\mathcal{F} = \{T_1, T_2, \dots, T_m\}$ denote the invariant sets of a GIFS $(g_e)_{e\in \Gamma}$ with vertex set $V=\{1,2,\dots,m\}$. We make the following additional assumptions on the GIFS:
\begin{enumerate}[label=\arabic*.]
\item  Each tile in $\cF$ has a nonempty interior.
\item  Each tile in $\cF$ satisfies the open set condition, i.e.\ each union in \eqref{GIFS} has no interior-overlap.
\item The directed graph $(V,\Gamma)$ is strongly connected, i.e.\ for all $i,j \in V$, there is similar copy of $T_i$ in the subdivision of $T_j$.
\end{enumerate}

\subsubsection{\textbf{$\boldsymbol{\varepsilon}$-rule}.}
For convenience, let $\Sigma_{N}$ denote the set of edge sequences of length $N$ on $(V,\Gamma)$.  In other words, $\Sigma_{N}$ is the collection of sequences $e_1,e_2,\dots, e_N \in \Gamma$ such that the composition  $g_{e_1} \circ g_{e_2} \circ \dots \circ g_{e_N}$ is permitted by the GIFS.
We iterate the GIFS starting from an initial tile $T_j\in \cF$ with the following \textit{$\varepsilon$-rule}:

\begin{enumerate}[label=(\roman*)]
\item Fix $0<\varepsilon< 1$.  Given a finite sequence $(e_1, e_2, \dots, e_{m(\varepsilon)}) \in \Sigma_{m(\varepsilon)}$, we
continue applying the GIFS to the tile
$$
T = g_{e_1}\circ g_{e_2} \circ \dots \circ g_{e_{m(\varepsilon)}}(T_j)\,,
$$
while $\text{Area}(T) > \varepsilon$, and stop applying the GIFS  to $T$ when $\text{Area}(T) \leq \varepsilon$.
\item Once the process in (i) terminates, we inflate the resulting collection of tiles by $\frac{1}{\sqrt{\varepsilon}}$.  We denote the resulting finite patch by $\mathcal{P}_{\varepsilon}(T_j)$.
\end{enumerate}

\begin{remark}\label{range_of_area}
Let $|g_e'|$ denote the contraction factor of the similitude $g_e$ in the GIFS. After applying the $\varepsilon$-rule, we are guaranteed to have $\Area(T)\in[a,1]$ for every tile $T\in\cP_{\varepsilon}(T_j)$, where
$$
a=\min\{|g_e'|^2\, :\, e \in \Gamma\} \,.
$$
	
	\end{remark}
	\subsubsection{\textbf{Associated point sets.}}
	Our next step is to define a point set $\Lambda_j(\varepsilon)$ associated with the patch $\cP_{\varepsilon}(T_j)$ by placing one point inside each tile. Since $\cF$ is a finite set, we can choose fixed constants $r_0\,,R_0>0$ which do not depend on $T_j$ and distinguished points
	$\lambda_j^{(0)}\in T_j$ for each $j \in V$ such that
	\begin{equation}\label{r_0R_0}
B(\lambda_j^{(0)},r_0) \subset T_j \subset B(\lambda_j^{(0)},R_0)\,.
\end{equation}
Let $0<\varepsilon<1$ and $T_j \in \cF$. Starting from $\lambda_j^{(0)}$, we construct the elements of $\Lambda_j(\varepsilon)$ recursively as follows.  For each application of the GIFS
$$
T_j=\bigcup\limits_{i\in V}\bigcup\limits_{e\in\Gamma_{j,i}}g_e(T_i)\,,\quad  j\in V \,,
$$
in the $\varepsilon$-rule, we produce new points
\[
\lambda_j^{(n)}=\bigcup\limits_{i\in V}\bigcup\limits_{e\in\Gamma_{j,i}}g_e(\lambda_i^{(n-1)})\,,\quad j\in V \,,
\]
that replace the previous points $\lambda_i^{(n-1)}$.  This process terminates when there are no further applications of the GIFS.  Finally, we inflate the resulting collection of points by $\frac{1}{\sqrt{\varepsilon}}$ to obtain \emph{the associated point sets} $\Lambda_j(\varepsilon)$.

In other words, since any tile $T$ in $\cP_{\varepsilon}(T_j)$ is similar to a unique $T_i\in \cF$, we place exactly one point in $T$ via similarity at the same location as $\lambda_i^{(0)}\in T_i$. Though one can choose any points $\{\lambda_j^{(0)}\}_{j \in V}$ such that \eqref{r_0R_0} holds, for simplicity we choose the centroid of each tile.  In this case, $\Lambda_j(\varepsilon)$ is simply the set of centroids of the tiles in $\cP_\varepsilon(T_j)$. Furthermore, we introduce the additional assumption that
\begin{align}\label{range_of_area_2}
\min\{\Area(T_j) \,:\, T_j\in \cF\}=1,\quad  \max\{\Area(T_j) \,:\, T_j\in \cF\}=S\geq 1 \,,
\end{align}
as this will simplify subsequent proofs.  The value $S$ is related to the  GIFS $(g_e)_{e \in \Gamma}$; however, one can easily modify a given GIFS to satisfy \eqref{range_of_area_2} 
by simultaneously enlarging or shrinking all tiles $T_j$ by a fixed similitude.

\subsubsection{\textbf{Existence of a Delone limit set}}
Next, we consider a decreasing sequence $\{\varepsilon_n\}_{n \geq 1}$ in $(0,1)$ and show that the corresponding sequence of point sets $\{\Lambda_j(\varepsilon_n)\}_{n \geq 1}$ has a subsequence converging to a Delone set in the Chabauty--Fell topology.

\begin{lem}\label{Limit_Delone}
\emph{
	For any decreasing sequence $\{\varepsilon_n\}_{n \geq 1}$ in $(0,1)$ with $\lim_{n\rightarrow \infty} \varepsilon_n = 0$, there exists a sequence $\{K_n\}_{n\geq1}$ of compact sets such that
	\begin{enumerate}[label=(\roman*)]
		\item $K_n \subseteq K_{n+1}$ for all $n\geq1$,
		\item $\cup_{n\geq 1} K_n = \RR^2$, and
		\item $\Lambda_j(\varepsilon_n)\in \cW_{r,R}(K_n)$ for all $n\geq1$, where $r=\sqrt{\tfrac{a}{S}}r_0$ and $R=R_0$.
	\end{enumerate}
	In other words, $\Lambda_j(\varepsilon_n)$ is an $(r,R)$-Delone set in $K_n$ for each $n\geq1$.
}
\end{lem}
\begin{proof}
Recall from \eqref{r_0R_0} that we have
\[
B(\lambda_j^{(0)},r_0) \subset T_j \subset B(\lambda_j^{(0)},R_0)\,, \quad \forall\, T_j\in\cF\,.
\]

Given $0<\varepsilon_n<1$, consider an arbitrary tile $T\in \cP_{\varepsilon_n}(T_j)$. There must exist a unique finite sequence $(e_1,e_2,\ldots,e_{m(\varepsilon_n)})\in\Sigma_{m(\varepsilon_n)}$ and one tile $T_{i_0}\in\cF$ such that
\[
T=\tfrac{1}{\sqrt{\varepsilon_n}}g_{e_1}\circ g_{e_2} \circ \dots \circ g_{e_{m(\varepsilon_n)}}(T_{i_0})\,.
\]
Define $x_T$ to be the point of $\Lambda_j(\varepsilon_n)$ that lies in $T$, i.e.
$$
x_T= \tfrac{1}{\sqrt{\varepsilon_n}}g_{e_1}\circ g_{e_2} \circ \dots \circ g_{e_{m(\varepsilon_n)}}(\lambda_{i_0}^{(0)})\,.
$$

Let $c=\Area(T)$ and $A=\Area(T_{i_0})$. Hence, the linear scaling factor from $T_{i_0}$ to $T$ is $\sqrt{\frac{c}{A}}$. Thus, scaling $B(\lambda_{i_0}^{(0)},r_0)$ by the linear factor $\sqrt{\frac{c}{A}}$ will result in a ball of radius $\sqrt{\frac{c}{A}}r_0$ that fits inside $T$ when centered at $x_T$. Hence, we have $B(x_T,\sqrt{\frac{c}{A}}r_0)\subseteq T$. Similarly, we have $T\subseteq B(x_T,\sqrt{\frac{c}{A}}R_0)$.   Moreover, by Remark~\ref{range_of_area}, we have $a \leq c\leq 1$. By assumption~\ref{range_of_area_2}, we get $A\in[1,S]$. From this, we obtain
\begin{equation}\label{eq:inclusions}
	B(x_T,\textstyle{\sqrt{\tfrac{a}{S}}}r_0) \subseteq B(x_T,\textstyle{\sqrt{\tfrac{c}{A}}}r_0) \subseteq T \subseteq B(x_T,\textstyle{\sqrt{\tfrac{c}{A}}}R_0) \subseteq B(x_T,R_0) \quad \forall \, T \in \cP_{\varepsilon_n}(T_j)\,.
\end{equation}

Assume without loss of generality that $\lambda_j^{(0)}=0$.  Then, for each $\varepsilon_n$, we have that
\begin{equation*}
	\Lambda_j(\varepsilon_n)\subseteq K_n = \bigcup_{T\in\cP_{\varepsilon_n}(T_j)}T \,.
\end{equation*}
For this choice of $K_n$, it is easy to see that conditions (i) and (ii) are satisfied.  To prove condition (iii), we must show that for every $x\in K_n$,  $\card(B(x,r)\cap\Lambda_j(\varepsilon_n)\leq 1$ and $\overline{B(x,R)}\cap \Lambda_j(\varepsilon_n)\neq \emptyset$, when $r=\sqrt{\tfrac{a}{S}}r_0$ and $R=R_0$.

To this end, let $x\in K_n$ be arbitrary. By our choice of $K_n$, there exists $T\in\cP_{\varepsilon_n}(T_j)$ such that $x\in T$.  First, notice that \eqref{eq:inclusions} gives us $x\in B(x_T,R)$, which implies that $x_T \in B(x,R)$, so $\overline{B(x,R)}\cap \Lambda_j(\varepsilon_n) \neq \emptyset$ holds.   It remains to prove that $\card(B(x,r)\cap  \Lambda_j(\varepsilon_n))\leq 1$. We consider two cases:

\textit{Case 1.}  Assume that $x$ lies in $B(x_T,r)$.  Then $x_T\in B(x,r)$. We claim that $x_{T'}\notin B(x,r)$ for any $T' \in \cP_{\varepsilon_n}(T_j) \backslash T$.  Indeed, we have
\begin{equation}\label{triangle_inequality}
	2r \leq \|x_T-x_{T'}\| \leq \|x_T-x\| + \|x-x_{T'}\| \leq r + \|x-x_{T'}\|\,.
\end{equation}
the first inequality holds since $T'$ and $T$ are non-overlapping, and by \eqref{triangle_inequality} we have $\|x-x_{T'}\| \geq r$.

\textit{Case 2.}  Assume that $x$ lies outside of $B(x_T,r)$, i.e.\ $x_T \notin B(x,r)$. We need to check that there can be at most one tile $T'\in\cP_{\varepsilon_n}(T_j) \backslash T$ such that $x_{T'}\in B(x,r)$.  On the contrary, suppose that there exist two distinct tiles $T',T''\in\cP_{\varepsilon_n}(T_j) \backslash T$ with $x_{T'},x_{T''}\in B(x,r)$.  Since $T'$ and $T''$ have no overlap, we must have $\|x_{T'}-x_{T''}\|\geq 2r$.  On the other hand, since $x_{T'},x_{T''}\in B(x,r)$, the triangle inequality gives us
\[
\|x_{T'}-x_{T''}\| \leq \|x_{T'}-x\| + \|x-x_{T''}\| < r+r=2r\,,
\]
a contradiction.

Thus, we have shown that $\card(B(x,r)\cap\Lambda_j(\varepsilon_n))\leq 1$ and $\overline{B(x,R)}\cap \Lambda_j(\varepsilon_n) \neq \emptyset$.  As $x$ was arbitrary, this completes the proof.
\end{proof}

\begin{thm}\label{thm:convergence}
\emph{
	There exists a subsequence of $\Lambda_j(\varepsilon_n)$  that converges to a $(r,R)$-Delone set in $\RR^2$ in the Chabauty--Fell topology.
}
\end{thm}
\begin{proof}
By Lemma~\ref{Limit_Delone}, we have $\Lambda_j(\varepsilon_n) \subseteq \mathcal{W}_{r,R}(K_n)$, so there must exist a point set $\Lambda_n \in \mathcal{W}_{r,R}(\RR^2)$ such that $\Lambda_n \cap K_n = \Lambda_j(\varepsilon_n)$.  Next, since $\cW_{r,R}(\RR^2)$ is compact, there exists a subsequence $\{\Lambda_{n_i}\}_{i \geq 1}$ of $\{\Lambda_n\}_{n \geq 1}$ converging to some $\Lambda \in \cW_{r,R}(\RR^2)$ in the Chabauty--Fell topology.  Moreover, we have $K_n \subseteq K_{n+1}$ for all $n$ and $\bigcup_{n\geq 1}K_n = \RR^2$. By compactness of $K_{n_i}$ in $\RR^2$ (in the Euclidean topology), we can pick a sequence $R_{n_i}$ of positive real numbers such that $K_{n_i} \subseteq B(0,R_{n_i})$ for all $i$.
Therefore, by Proposition~\ref{prop:Fell-compact-convergence}, we also have that $\Lambda_{n_i} \cap B(0,R_{n_i})$ converges to $\Lambda$ in the Chabauty--Fell topology. This proves the theorem.
\end{proof}

\section{Tilings By Badly Approximable Triangles}\label{sec:specific}

In this section, we consider the subdivision rule given in Figure~\ref{fig:subdivision-rule-new} for arbitrary angles $\alpha,\beta,\gamma$. Below, we describe the GIFS associated with this subdivision rule. Then, as illustrated in Section~\ref{sec:optimal-tilings}, we can use this GIFS and the methods in Section~\ref{sec:threshold} to produce a sequence of finite patches leading to a Delone set.

\subsection{GIFS for arbitrary angles}
\label{GIFS2}

Let $T_1$ and $T_2$ be the scalene and isosceles triangle of unit area, respectively. We compute the GIFS for $\cF=\{T_1,T_2\}$.  We use $e^{i\theta}$ to denote rotation about the origin by the angle $\theta$. Suppose that the bottom left corners of the tiles are at the origin. The contractive similitudes are as follows:
\begin{equation}\label{eq:GIFS-specific}
\begin{aligned}
	& f_1(x,y)=\tfrac{1}{1+t^2}(x,y)\,, \\
	& f_2(x,y)=\tfrac{t}{s(1+t^2)}(-x,y) \cdot e^{i(\pi-\beta)}+(a,0)+at(\cos\gamma,\sin\gamma)\,,\\
	& f_3(x,y)=\tfrac{t}{1+t^2}(x,y)\cdot e^{i\gamma}+(a,0)\,,\\ &f_4(x,y)=\tfrac{\sqrt{C}t}{1+t^2}(-x,y) \cdot e^{i\alpha}+(bu,0)\,, \\
	& f_5(z)=\tfrac{\sqrt{C}t}{1+t^2}(-x,y) \cdot e^{i(\pi-\alpha)}+(bu,0)+bt(\cos\alpha,\sin\alpha) \,, \\
	& g_1(x,y)=\tfrac{u}{\sqrt{C}(st+u)}(x,y) \cdot e^{i(\pi-\beta)}+(a,0)+at(\cos(\gamma),\sin(\gamma))\,,\\
	& g_2(x,y)=\tfrac{u}{st+u}(x,y)\,, \\
	& g_3(x,y)=\tfrac{st}{st+u}(x,y)+bst^2(\cos\gamma,\sin\gamma)\,,
\end{aligned}
\end{equation}
where
\begin{equation*}
s = \tfrac{\sin\beta}{\sin\gamma}
\,,\quad t = \tfrac{\sin\alpha}{\sin\beta}
\,,\quad u = 2st^2\cos\gamma=\tfrac{2\sin^2\gamma}{\sin\beta\tan\gamma}
\,, \quad C=\tfrac{2(1+t^2)^2\sin\alpha\cot \gamma}{t(st+u)(1+2t\cos \gamma)}\,,
\end{equation*}
and
\begin{equation*}
a = \tfrac{1}{1+t^2}\sqrt{\tfrac{2}{s\sin(\alpha)}}\,,
\quad b = 2\sqrt{\tfrac{\cot(\gamma)}{st(st+u)(2t\cos(\gamma) + 1)}} \,.
\end{equation*}
Then we get
\begin{equation*}
T_1=g_1(T_2) \cup f_1(T_1) \cup f_2(T_1) \cup f_3(T_1) \,,
\end{equation*}
and
\begin{equation*}
T_2=g_2(T_2) \cup g_3(T_2) \cup f_4(T_1) \cup f_5(T_1)\,.
\end{equation*}

\subsection{Tilings with optimal badly approximable angles}\label{sec:optimal-tilings}

Our objective is to obtain Delone sets associated with the unique solutions to the equality $x+y+z=1\ (x,y,z\in\B_{2,1},\ x \leq y \leq z)$ found in Theorem~\ref{Main}:
\begin{equation*}
x=2-\sqrt{3}\,,\quad y=z=\frac{\sqrt{3}-1}{2}\,,
\end{equation*}
and
\begin{equation*}
x=y=\frac{2-\sqrt{2}}{2}\,,\quad  z=\sqrt{2}-1\,.
\end{equation*}
For the first solution, by choosing
\begin{equation}\label{eq:optimal1-angles}
\alpha = (2-\sqrt{3})\pi\,, \quad \beta=\gamma=\frac{(\sqrt{3}-1)\pi}{2} \,,
\end{equation}
the subdivision rule in Figure~\ref{fig:subdivision-rule-new} reduces to a subdivision rule involving only the isosceles triangle with angles $\alpha$ and $\beta=\gamma$  as shown in Figure~\ref{fig:subdivision-rule-iso}.
\begin{figure}[ht]
\centering
\includegraphics{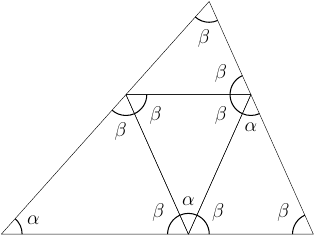}
\caption{Subdivision rule for the isosceles triangle with optimal badly approximable angles $\alpha$ and $\beta=\gamma$ as in \eqref{eq:optimal1-angles}.  }
\label{fig:subdivision-rule-iso}
\end{figure}
Similarly, for the second solution, by choosing
\begin{equation}\label{eq:optimal2-angles}
\alpha = (\sqrt{2}-1)\pi\,, \quad \beta=\gamma=\frac{(2-\sqrt{2})\pi}{2} \,,
\end{equation}
we again get a subdivision rule involving only an isosceles triangle, exactly as in Figure~\ref{fig:subdivision-rule-iso} except with different values for $\alpha$ and $\beta$.

In this section, we provide some illustrations of our construction in Section~\ref{GIFS2} in the specific case when $\alpha$ and $\beta=\gamma$ are as in $\eqref{eq:optimal1-angles}$ and \eqref{eq:optimal2-angles}. First, in Figure~\ref{fig:optimal_iterations}, we illustrate the $\varepsilon$-rule process for $\varepsilon=0.2$ as an example.  Then, in Figure~\ref{fig:optimal1_patches} and Figure~\ref{fig:optimal2_patches}, we show the final patches for a few different values of $\varepsilon$.

\begin{figure}[ht]
\begin{subfigure}[c]{\textwidth}
	\centering
	\includegraphics{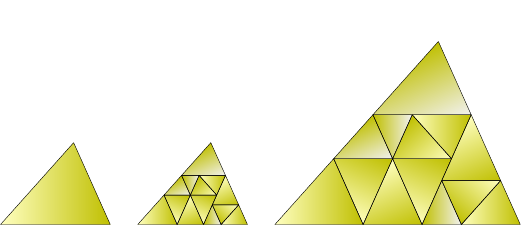}
	\caption{$\alpha=(2-\sqrt{3})\pi$, $\beta=\gamma=\frac{(\sqrt{3}-1)\pi}{2}$}
\end{subfigure}
\par\bigskip
\begin{subfigure}[c]{\textwidth}
	\centering
	\includegraphics{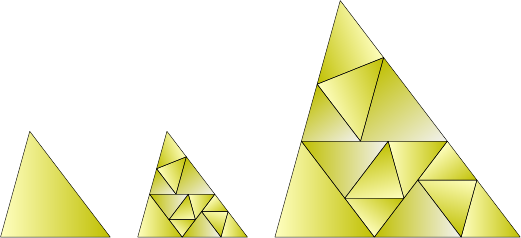}
	\caption{$\alpha=(\sqrt{2}-1)\pi$, $\beta=\gamma=\frac{(2-\sqrt{2})\pi}{2}$}
\end{subfigure}
\caption{Illustration of the $\varepsilon$-rule for $\varepsilon=0.2$ for optimal badly approximable angles.}
\label{fig:optimal_iterations}
\end{figure}
\begin{figure}[ht]
\begin{subfigure}[b]{0.3\textwidth}
	\centering
	\includegraphics{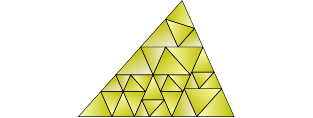}
	\caption{$\varepsilon=0.08$}
\end{subfigure}
\hskip -9ex
\begin{subfigure}[b]{0.3\textwidth}
	\centering
	\includegraphics{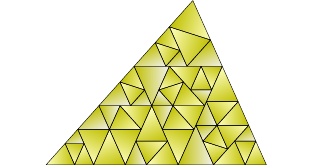}
	\caption{$\varepsilon=0.04$}
\end{subfigure}
\begin{subfigure}[b]{0.3\textwidth}
	\centering
	\includegraphics{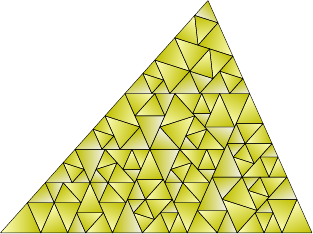}
	\caption{$\varepsilon=0.02$}
\end{subfigure}
\caption{Final patches resulting from different $\varepsilon$-rules for optimal badly approximable angles $\alpha=(2-\sqrt{3})\pi$, $\beta=\gamma=\frac{(\sqrt{3}-1)\pi}{2}$.}
\label{fig:optimal1_patches}
\end{figure}
\begin{figure}[ht]
\begin{subfigure}[b]{0.3\textwidth}
	\centering
	\includegraphics{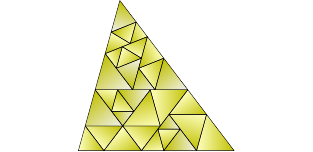}
	\caption{$\varepsilon=0.08$}
\end{subfigure}
\hskip -9ex
\begin{subfigure}[b]{0.3\textwidth}
	\centering
	\includegraphics{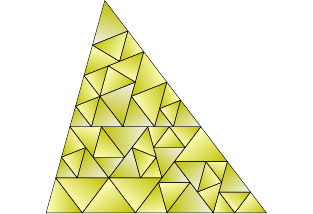}
	\caption{$\varepsilon=0.04$}
\end{subfigure}
\begin{subfigure}[b]{0.3\textwidth}
	\centering
	\includegraphics{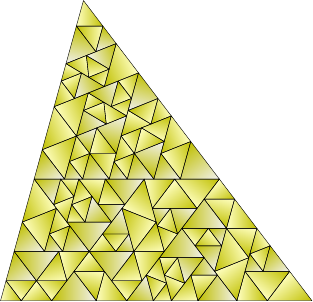}
	\caption{$\varepsilon=0.02$}
\end{subfigure}
\caption{Final patches resulting from different $\varepsilon$-rules for optimal badly approximable angles $\alpha=(\sqrt{2}-1)\pi$, $\beta=\gamma=\frac{(2-\sqrt{2})\pi}{2}$.}
\label{fig:optimal2_patches}
\end{figure}

As described in \cite[Remark~4.11]{Smi-Solo:21}, one can produce a stationary tiling by applying an additional isometry after each $\varepsilon$-rule.  For our subdivision rule shown in Figure~\ref{fig:subdivision-rule-new}, we can do this by introducing a rotation by the badly approximable angle $-\gamma$.

To make this more precise, choose the sequence $\varepsilon_n = \varepsilon_0^n$ where $\varepsilon_0 = \frac{t^2}{(1+t^2)^2}$ and $t = \frac{\sin(\alpha)}{\sin(\beta)}$.  Then $\varepsilon_0$ is the square of the contraction factor of the similitude $f_3$ in \eqref{eq:GIFS-specific}.  This function maps the large scalene triangle $T_1$ shown on the left in Figure~\ref{fig:subdivision-rule-new} to the smaller scalene triangle in its center.  For this value of $\varepsilon_0$, the finite patch $\cP_{\varepsilon_0}(T_1)$ contains a copy of $e^{i\gamma}T_1$.  With this choice of $\varepsilon_n$,  the sequence $e^{-in\gamma}(\cP_{\varepsilon_n}(T_1))$ will produce a stationary tiling.  Figure~\ref{fig:optimal_stationary} shows the resulting nested sequence of finite patches when $\alpha$ and $\beta=\gamma$ are as in $\eqref{eq:optimal1-angles}$ and \eqref{eq:optimal2-angles}.

\begin{figure}[ht]
\begin{subfigure}[c]{\textwidth}
	\centering
	\includegraphics{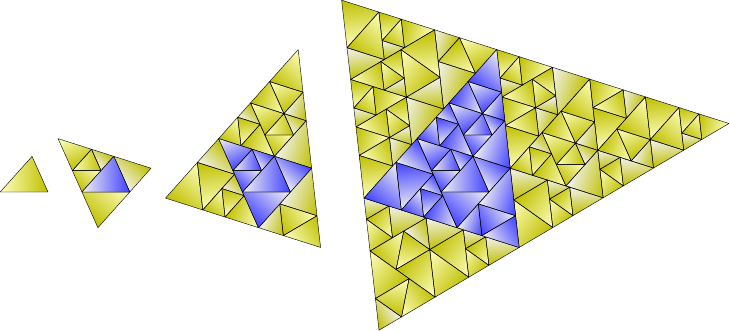}
	\caption{$\alpha=(2-\sqrt{3})\pi$, $\beta=\gamma=\frac{(\sqrt{3}-1)\pi}{2}$}
\end{subfigure}
\par\bigskip
\begin{subfigure}[c]{\textwidth}
	\centering
	\includegraphics{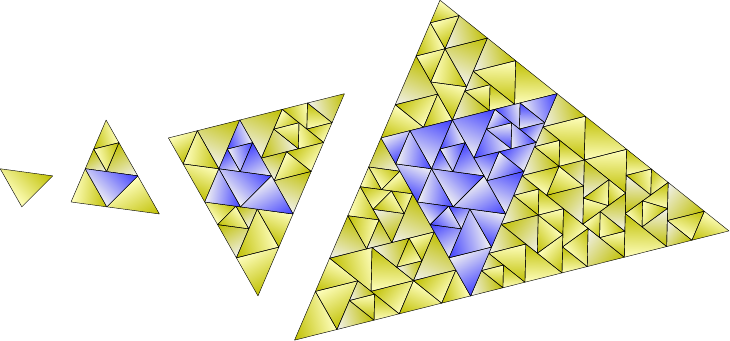}
	\caption{$\alpha=(\sqrt{2}-1)\pi$, $\beta=\gamma=\frac{(2-\sqrt{2})\pi}{2}$}
\end{subfigure}
\caption{Illustration of the sequences producing stationary tilings for badly approximable angles. The stationary tiling is produced by choosing $\varepsilon_n = \varepsilon_0^n$ where $\varepsilon_0$ such that the area of the middle triangle is preserved, and applying an appropriate isometry after applying each $\varepsilon_n$-rule. Each stationary subpatch is highlighted in blue. In (B), an additional rotation by $\frac{\pi}{4}$ is applied for ease of illustration. }
\label{fig:optimal_stationary}
\end{figure}

Each finite patch in these nested sequences contains the previous finite patch, along with several new copies of the triangle rotated by badly approximable angles $\alpha$ and $\beta=\gamma$.  Due to these irrational rotations, we expect the diffraction from the resulting stationary tilings to be rotationally invariant. Moreover, the discrepancy of each irrational rotation is small and optimal for our two choices of angles.

We conclude this section with a remark on convergence.

\begin{remark}
The tilings produced by our method do not have finite local complexity; because of this, we utilized the Chabauty--Fell topology to show the existence of a limit Delone set via a non-constructive process of subsequence selection.  However, by the method described above, we can always produce a nested sequence of finite patches converging to a stationary tiling.  In this case, we also have convergence in the \emph{local topology}, i.e.\ the finite patches overlap exactly out to larger and larger radii up to small translations (see \cite[Chapter~5]{Baake-Grimm:13} for a formal definition). This stronger notion of convergence is primarily used in the study of tilings with finite local complexity, but is still applicable stationary constructions.
\end{remark}

\section{Proof of Theorems~\ref{Main}, \ref{Main2} and \ref{Main3}}\label{sec:proof_main123}

To prove Theorem~\ref{Main} and \ref{Main2}, we will need the following lemmas:
\begin{lem}[Forbidden patterns]\label{lem:x3-y3-z-1}
\emph{
	Let $z\in(0,1)\backslash\QQ$ be such that
	$z \in [r_1, r_2]$
	for some $r_1,r_2\in(0,1)\cap \QQ$.  If the simple continued fractions of $r_1$ and $r_2$ are of the form
	\begin{equation}\label{eq:forbidden1}
		r_1 = [(2)^{2k-1},(2,1)^\ell,s,a_1,\dots,a_{n_1},\infty]\,,\quad r_2 = [(2)^{2k-1},b_1,\dots,b_{n_2},\infty]\,,
	\end{equation}
	or
	\begin{equation}\label{eq:forbidden2}
		r_1 = [(2)^{2k},a_1,\dots,a_{n_1},\infty]\,,\quad r_2 = [(2)^{2k},(2,1)^\ell,s,b_1,\dots,b_{n_2},\infty]\,,
	\end{equation}
	where $k\geq 1$, $\ell \geq 0$, and $s\geq 3$, then $z \notin \B_2$.
}
\end{lem}

\begin{proof}

\medskip
\noindent \textit{Case 1.} Assume that $r_1$ and $r_2$ are of the form \eqref{eq:forbidden1}
where $k\geq 1$, $\ell \geq 0$, and $s\geq 3$.
\medskip

Let $z = [c_1,c_2,\dots]$. Suppose by contradiction that $z\in\B_2$. First, for simplicity, observe that \eqref{eq:forbidden1} implies
\begin{equation}\label{eq:forbidden1-simple}
	[(2)^{2k-1},(2,1)^\ell,s,\infty] \leq z \leq [(2)^{2k-1},\infty]\,.
\end{equation}
Since \eqref{eq:forbidden1} requires $c_1,\dots,c_{2k-1} \geq 2$ and $z\in\B_2$ requires $c_1,\dots,c_{2k-1} \leq 2$, we must have
\begin{equation}\label{eq:equals2}
	c_1=\dots=c_{2k-1}=2\,,
\end{equation}
i.e.\ $z\in I((2)^{2k-1})$.

Next, consider $u=[c_{2k},c_{2k+1},\dots]$. By \eqref{eq:forbidden1-simple} and \eqref{eq:equals2}, we have
\begin{equation}\label{eq:leq21}
	u \leq [(2,1)^\ell,s,\infty]\,.
\end{equation}
From this, we can deduce that $u\in I((2,1)^\ell)$. To see this, observe that \eqref{eq:leq21} gives us $c_{2k} \geq 2$, so we must have $c_{2k}=2$ because $u\in \B_2$.  We also get that $c_{2k+1} \leq 1$, so we must have $c_{2k+1} = 1$.  Proceeding inductively, we find that $c_{2q}=2$  and  $c_{2q+1}=1$ for $k\leq q\leq k+\ell-1$.

Now consider $v = [c_{2(k+\ell)},c_{2(k+\ell)+1},\dots]$.  From \eqref{eq:leq21} we get
\[
v \leq [s,\infty] = \frac{1}{s}\,,
\]
so $c_{2(k+\ell)} = \lfloor 1/v \rfloor \geq s \geq 3$, contradicting $v\in\B_2$. Therefore, $z$ cannot have the form \eqref{eq:forbidden1}.

\medskip
\noindent \textit{Case 2.} The proof when $r_1$ and $r_2$ are of the form \eqref{eq:forbidden2} is similar to Case 1.
\end{proof}

We call any interval of the form \eqref{eq:forbidden1} or \eqref{eq:forbidden2} a \textit{forbidden pattern}.  Moreover, if $z \in [r_1, r_2]$ for some forbidden pattern $[r_1, r_2]$, we say that $z$ is \emph{contained in a forbidden pattern.}

\begin{lem}\label{xyz}
\emph{
	The equality
	$
	x+y=z\ (x,y,z\in \B_{2})
	$
	has exactly one solution
	\begin{equation}\label{eq:B2-unique}
		x=y=\frac{\sqrt{3}-1}2=[\overline{2,1}],\ z=\sqrt{3}-1=[\overline{1,2}]\,.
	\end{equation}
}
\end{lem}
\begin{proof}
Suppose  $x=[a_1,a_2,\cdots]$ where $a_i\in \{1,2\}$ with $i\geq 1$ and $y=[b_1,b_2,\cdots]$ where $b_j\in \{1,2\}$ with $j\geq 1$. It is very easy to get that
$x<y$ if and only if $(-1)^na_n<(-1)^nb_n$ for the first $n$ such that $a_n\neq b_n$.
As a consequence, we have
\[
\frac{\sqrt{3}-1}{2}=[\overline{2,1}] \,  = \min(\B_2)\,,
\]
and
\[
\sqrt{3}-1=[\overline{1,2}] \, = \max(\B_2)\,.
\]
With this preparation, we show that \eqref{eq:B2-unique} is the unique solution of
$
x+y=z\ (x,y,z\in \B_{2})
$ by considering the sum of $x$ and $y$.

If $x+y<\sqrt{3}-1$, we get that one of them is less than $\frac{\sqrt{3}-1}{2}$, which contradicts $\frac{\sqrt{3}-1}{2}$ being the minimum of $\B_2$. Thus, $x+y\geq \sqrt{3}-1$. Similarly, if $x+y>\sqrt{3}-1$, then $z=x+y>\sqrt{3}-1$, which contradicts $\sqrt{3}-1$ being the maximum of $\B_2$.

Therefore, $z=x+y=\sqrt{3}-1$ and $x=y=\frac{\sqrt{3}-1}{2}$.
\end{proof}
We can now provide a concise proof of Theorem~\ref{Main}.  Indeed, we will see that the main part of the proof can be summarized by the careful analysis of cases shown in Table~\ref{tab:my_label_1} and Table~\ref{tab:my_label_2}, which we present below.
However, the reader may also wish to refer to \eqref{eq_1_X_Y_Zb} and Remark~\ref{rem:intuition} for an intuitive explanation of why this case analysis is stable as the parameter $n$ increases.  

\subsection{Proof of Theorem~\ref{Main}}
Assume that $x \leq y \leq z$. We divide the proof into 4 cases:

\medskip
\noindent \textit{Case 1.} $x,y,z \in \B_2$.
\medskip

\noindent This case is impossible by Remark~\ref{rem:no-B2-solution}.

\medskip
\noindent \textit{Case 2.} $x\in \B_{2,1}\setminus\B_2$, and $y,z\in\B_2$.
\medskip

\noindent Setting $x=\frac{1}{3+X}$ with $X\in\B_2$, we have
$$
y+z=1-\frac{1}{3+X}=\frac{1}{1+\frac{1}{2+X}}\in \B_2\,.
$$
By Lemma~\ref{xyz}, we obtain
$$
y=z=\frac{\sqrt{3}-1}{2}\,, \quad 1-\frac{1}{3+X}=\sqrt{3}-1\,,
$$
i.e.
\begin{align}\label{x+y+z=1_1}
x=\frac{1}{3+X}=2-\sqrt{3}=[3,\overline{1,2}]\,, \quad y=z=\frac{\sqrt{3}-1}{2}=[\overline{2,1}]\,,
\end{align}
is the solution of $x+y+z=1$.

\medskip
\noindent \textit{Case 3.} $x,y \in \B_{2,1}\setminus\B_2$, and $z\in\B_2$.
\medskip

\noindent In this case, the continued fraction expansions of $x$ and $y$ both satisfy $a_1 = 3$, $a_j \leq 2$ for all $j \geq 2$. Under these assumptions, we aim to show that
\begin{equation*}
x=y=\frac{2-\sqrt{2}}{2} = [3,\overline{2}]\,, \quad z = \sqrt{2}-1 = [\overline{2}]\,,
\end{equation*}
is the only possible solution. First, we introduce variables $X_n,Y_n,Z_n$ to represent arbitrary numbers satisfying a certain property that depends on a given nonnegative integer $n$.  More specifically, we will prove the following statement:
\medskip

For each $n\geq 0$\,, if
$X_n \notin I(3,(2)^{n+1})$ or\
$Y_n \notin I(3,(2)^{n+1})$, then $Z_n=1-X_n-Y_n$ is contained in a forbidden pattern.
\medskip

\noindent We divide the proof into 2 cases according to the parity of $n$.

\medskip
\noindent \textit{Case 3.1.} $n$ is even.
\medskip

\noindent The computations when $n$ is even are summarized in Table~\ref{tab:my_label_1}.  There are 17 cases that must be considered for the cylinder sets of $X_n$ and $Y_n$. 
\begin{table}[ht]
\centering
\footnotesize
\begin{tabular}{c|c|c|c|c}
	\thead{Case} &
	\thead{Cylinder set \\ for $X_n$} & \thead{Cylinder set \\ for $Y_n$} & \thead{Left endpoint of \\ the forbidden pattern} & \thead{Right endpoint of \\ the forbidden pattern}\\
	\hline
	1 & $I(3,(2)^n,1)$ & $I(3,(2)^n,1)$ &  $[(2)^{n+1},3,\infty]$ & $[(2)^{n+1},\infty]$   \\
	2.1 & $I(3,(2)^n,1,2)$ & $I(3,(2)^n,2,2)$  &  $[(2)^{n+1},3,\infty]$ & $[(2)^{n+1},3,1,\infty]$   \\
	2.2&$I(3,(2)^n,1,2)$ & $I(3,(2)^n,2,1)$   & $[(2)^{n+1},(2,1),12,\infty]$ & $[(2)^{n+1},3,1,\infty]$  \\
	2.3.1&$I(3,(2)^n,1,1,1)$ & $I(3,(2)^n,2,2,1)$ & $[(2)^{n+1},3,\infty]$ & $[(2)^{n+1},3,2,\infty]$  \\
	2.3.2&$I(3,(2)^n,1,1,1)$ & $I(3,(2)^n,2,2,2)$  & $[(2)^{n+1},3,\infty]$ & $[(2)^{n+1},3,3,\infty]$  \\
	2.3.3&$I(3,(2)^n,1,1,2)$ & $I(3,(2)^n,2,2,1)$  & $[(2)^{n+1},(2,1),14,\infty]$  & $[(2)^{n+1},3,10,\infty]$ \\
	2.3.4&$I(3,(2)^n,1,1,2)$ & $I(3,(2)^n,2,2,2)$  & $[(2)^{n+1},(2,1),10,\infty]$ & $[(2)^{n+1},3,20,\infty]$  \\
	2.4.1&$I(3,(2)^n,1,1,1)$ & $I(3,(2)^n,2,1,1)$  & $[(2)^{n+1},(2,1),6,\infty]$  & $[(2)^{n+1},3,4,\infty]$ \\
	2.4.2&$I(3,(2)^n,1,1,1)$ & $I(3,(2)^n,2,1,2)$  & $[(2)^{n+1},(2,1),4,\infty]$ & $[(2)^{n+1},3,8,\infty]$  \\
	2.4.3&$I(3,(2)^n,1,1,2)$ & $I(3,(2)^n,2,1,1)$  & $[(2)^{n+1},(2,1),3,\infty]$ & $[(2)^{n+1},2,1,\infty]$  \\
	2.4.4.1.1&$I(3,(2)^n,1,1,2,1,1)$ & $I(3,(2)^n,2,1,2,1,1)$  & $[(2)^{n+1},(2,1),3,\infty]$ & $[(2)^{n+1},2,1,\infty]$  \\
	2.4.4.1.2&$I(3,(2)^n,1,1,2,1,1)$ & $I(3,(2)^n,2,1,2,1,2)$  & $[(2)^{n+1},(2,1),3,\infty]$  & $[(2)^{n+1},2,1,\infty]$ \\
	2.4.4.1.3&$I(3,(2)^n,1,1,2,1,2)$ & $I(3,(2)^n,2,1,2,1,2)$  & $[(2)^{n+1},(2,1),3,\infty]$ & $[(2)^{n+1},2,1,\infty]$  \\
	2.4.4.1.4&$I(3,(2)^n,1,1,2,1,2)$ & $I(3,(2)^n,2,1,2,1,1)$  & $[(2)^{n+1},(2,1),3,\infty]$ & $[(2)^{n+1},2,1,\infty]$  \\
	2.4.4.2&$I(3,(2)^n,1,1,2,1)$ & $I(3,(2)^n,2,1,2,2)$  & $[(2)^{n+1},(2,1)^2,20,\infty]$ & $[(2)^{n+1},2,1,\infty]$  \\
	2.4.4.3&$I(3,(2)^n,1,1,2,2)$ & $I(3,(2)^n,2,1,2,1)$  & $[(2)^{n+1},(2,1),3,\infty]$ & $[(2)^{n+1},2,1,\infty]$  \\
	2.4.4.4&$I(3,(2)^n,1,1,2,2)$ & $I(3,(2)^n,2,1,2,2)$  & $[(2)^{n+1},(2,1),3,\infty]$& $[(2)^{n+1},2,1,\infty]$  \\
\end{tabular}
\caption{Case analysis showing the forbidden patterns containing $Z_n=1-X_n-Y_n$ when} $n$ is even.
\label{tab:my_label_1}
\end{table}
The computations are lengthy and we only detail Case~2.1 of Table~\ref{tab:my_label_1} in this paper.

To prove the result for  Case~2.1 in Table~\ref{tab:my_label_1}, assume that
\begin{equation*}
X_n \in I(3,(2)^n,1,2)\,, \quad Y_n \in I(3,(2)^n,2,2) \,.
\end{equation*}
Setting
\[
A_n=1-\tfrac{\sqrt{2}}{2}+\tfrac{12-19\sqrt{2}}{-19+6\sqrt{2}+17(1+\sqrt{2})^{2n+4}} \,, \quad
B_n=1-\tfrac{\sqrt{2}}{2}+\tfrac{10+\sqrt{2}}{1+5\sqrt{2}-7(1+\sqrt{2})^{2n+5}} \,,
\]
and
\[
C_n=1-\tfrac{\sqrt{2}}{2}+\tfrac{\sqrt{2}}{-(1+\sqrt{2})^{2n+8}+1} \,, \quad
D_n=1-\tfrac{\sqrt{2}}{2}+\tfrac{\sqrt{2}}{(1+\sqrt{2})^{2n+8}+1} \,,
\]
we obtain
\[
I(3,(2)^n,1,2) =(A_n,B_n]=1-\tfrac{\sqrt{2}}{2}+\left(\tfrac{12-19\sqrt{2}}{-19+6\sqrt{2}+17(1+\sqrt{2})^{2n+4}},
\tfrac{10+\sqrt{2}}{1+5\sqrt{2}-7(1+\sqrt{2})^{2n+5}}\right]\,,
\]
\[
\ I(3,(2)^n,2,2)=(C_n,D_n]=1-\tfrac{\sqrt{2}}{2}+\left(\tfrac{\sqrt{2}}{-(1+\sqrt{2})^{2n+8}+1},
\tfrac{\sqrt{2}}{(1+\sqrt{2})^{2n+8}+1}\right] \,.
\]
We claim that
\begin{align*}
Z_n=1-X_n-Y_n\in \left[[(2)^{n+1},3,\infty]\,,[(2)^{n+1},3,1,\infty]\right] \,.
\end{align*}
Since $n+1$ is odd, by Lemma~\ref{lem:x3-y3-z-1}, this implies that the bounds on $Z_n$ are contained in a forbidden pattern.
Indeed,
\begin{equation*}
Z_n= 1-X_n-Y_n\in[1-B_n-D_n, 1-A_n-C_n) \,,
\end{equation*}
by direct computation, we get
\[
1-B_n-D_n=\sqrt2-1-\tfrac{10+\sqrt{2}}{1+5\sqrt{2}-7(1+\sqrt{2})^{2n+5}}-\tfrac{\sqrt{2}}{(1+\sqrt{2})^{2n+8}+1}\,,
\]
\[
1-A_n-C_n=\sqrt2-1-\tfrac{12-19\sqrt{2}}{-19+6\sqrt{2}+17(1+\sqrt{2})^{2n+4}}-\tfrac{\sqrt{2}}{-(1+\sqrt{2})^{2n+8}+1}\,.
\]
Since $n$ is even, 
\begin{align*}
&\ \ \ \ (1-B_n-D_n)-[(2)^{n+1},3,\infty]\\
&=\sqrt2-1-\tfrac{10+\sqrt{2}}{1+5\sqrt{2}-7(1+\sqrt{2})^{2n+5}}-\tfrac{\sqrt{2}}{(1+\sqrt{2})^{2n+8}+1}-
\tfrac{(-3+2\sqrt{2})^{n+2}+1}{(1-\sqrt{2})(-3+2\sqrt{2})^{n+2}+1+\sqrt{2} }\\
&=\tfrac{(1+\sqrt{2})^{n}(24+16\sqrt{2})-(1-\sqrt{2})^{n}(-24+16\sqrt{2})}
{\left(-(1-\sqrt{2})^{n+3}-(1+\sqrt{2})^{n+3}\right)\left(10+(2\sqrt{2}-1)(1-\sqrt{2})^{2n+6}-(2\sqrt{2}+1)(1+\sqrt{2})^{2n+6}\right)} \,,
\end{align*}
is positive. 
We show this by checking that the numerator and the denominator are both positive.  
Indeed, since $\left|1+\sqrt{2}\right|>\left|1-\sqrt{2}\right|$ and $\left|24+16\sqrt{2}\right|>\left|-24+16\sqrt{2}\right|$, the numerator is positive.  Similarly, the first term in the denominator $-(1-\sqrt{2})^{n+3}-(1+\sqrt{2})^{n+3}$ is negative since $\left|1-\sqrt{2}\right|<\left|1+\sqrt{2}\right|$. Now, since $\left|2\sqrt{2}-1\right|<\left|2\sqrt{2}+1\right|$ and $\left|1-\sqrt{2}\right|<\left|1+\sqrt{2}\right|$, the second term in the denominator
$$
10+(2\sqrt{2}-1)(1-\sqrt{2})^{2n+6}-(2\sqrt{2}+1)(1+\sqrt{2})^{2n+6} \,,
$$
is negative, so the denominator is positive. Therefore, the quotient is positive, as desired.

Next, we verify that
\begin{align*}
&\ \ \ \ \ [(2)^{n+1},3,1,\infty]-(1-A_n-C_n)\\
&=\tfrac{\frac{1}{7}(9+4\sqrt{2})(-3+2\sqrt{2})^{n+2}+1}{\frac{1}{7}(1- 5\sqrt{2})(-3+2\sqrt{2})^{n+2}+1+\sqrt{2}}-\left(\sqrt2-1-\tfrac{12-19\sqrt{2}}{-19+6\sqrt{2}+17(1+\sqrt{2})^{2n+4}}-\tfrac{\sqrt{2}}{-(1+\sqrt{2})^{2n+8}+1}\right)\\
&=\tfrac{(1+\sqrt{2})^{n}(72+60\sqrt{2})-(1-\sqrt{2})^{n}(-72+60\sqrt{2})}
{\left((1-\sqrt{2})^{n+3}(4+\sqrt{2})+(1+\sqrt{2})^{n+3}(4-\sqrt{2})\right)\left(-28+(1-\sqrt{2})^{2n+6}(6-\sqrt{2})+(1+\sqrt{2})^{2n+6}(6+\sqrt{2})\right)} \,,
\end{align*}
is positive.  As before, we can see this by checking that the numerator and the denominator are both positive.  This time, we have $\left|1+\sqrt{2}\right|>\left|1-\sqrt{2}\right|$ and $\left|72+60\sqrt{2}\right|>\left|72-60\sqrt{2}\right|$, so the numerator is positive.  Similarly, it is easy to check that the first term and the second term in the denominator are positive, so the denominator is positive. Therefore, the quotient is positive, as desired.
Thus, we obtain that
$$
Z_n=1-X_n-Y_n \in \left[[(2)^{n+1},3,\infty],[(2)^{n+1},3,1,\infty]\right] \,,
$$
which implies that the bounds on $Z_n$ are contained in a forbidden pattern.  The other cases in Table~\ref{tab:my_label_1} can be proven in the same way.

\medskip
\noindent    \textit{Case 3.2.} $n$ is odd.
\medskip

\noindent The computations of forbidden patterns containing $Z_n=1-X_n-Y_n$ when $n$ is odd are summarized in Table~\ref{tab:my_label_2}, where $X_n \notin I(3,(2)^{n+1})$ or\ $Y_n \notin I(3,(2)^{n+1})$.
Again, there are 17 cases that must be considered. The proofs are similar to those for the even case. 
\begin{table}[ht]
\centering
\footnotesize
\begin{tabular}{c|c|c|c|c}
	\thead{Case} &
	\thead{Cylinder set \\ for $X_n$} & \thead{Cylinder set \\ for $Y_n$} & \thead{Left endpoint of \\ the forbidden pattern} & \thead{Right endpoint of \\ the forbidden pattern}\\
	\hline
	1 & $I(3,(2)^n,1)$ & $I(3,(2)^n,1)$ & $[(2)^{n+1},\infty]$  &  $[(2)^{n+1},3,\infty]$ \\
	2.1 & $I(3,(2)^n,1,2)$ & $I(3,(2)^n,2,2)$  & $[(2)^{n+1},3,1,\infty]$  &  $[(2)^{n+1},3,\infty]$ \\
	2.2&$I(3,(2)^n,1,2)$ & $I(3,(2)^n,2,1)$    & $[(2)^{n+1},3,1,\infty]$ & $[(2)^{n+1},(2,1),12,\infty]$ \\
	2.3.1&$I(3,(2)^n,1,1,1)$ & $I(3,(2)^n,2,2,1)$  & $[(2)^{n+1},3,2,\infty]$ & $[(2)^{n+1},3,\infty]$ \\
	2.3.2&$I(3,(2)^n,1,1,1)$ & $I(3,(2)^n,2,2,2)$  & $[(2)^{n+1},3,3,\infty]$ & $[(2)^{n+1},3,\infty]$ \\
	2.3.3&$I(3,(2)^n,1,1,2)$ & $I(3,(2)^n,2,2,1)$  & $[(2)^{n+1},3,10,\infty]$ & $[(2)^{n+1},(2,1),14,\infty]$ \\
	2.3.4&$I(3,(2)^n,1,1,2)$ & $I(3,(2)^n,2,2,2)$  & $[(2)^{n+1},3,20,\infty]$ & $[(2)^{n+1},(2,1),10,\infty]$ \\
	2.4.1&$I(3,(2)^n,1,1,1)$ & $I(3,(2)^n,2,1,1)$  & $[(2)^{n+1},3,4,\infty]$ & $[(2)^{n+1},(2,1),6,\infty]$ \\
	2.4.2&$I(3,(2)^n,1,1,1)$ & $I(3,(2)^n,2,1,2)$  & $[(2)^{n+1},3,8,\infty]$ & $[(2)^{n+1},(2,1),4,\infty]$ \\
	2.4.3&$I(3,(2)^n,1,1,2)$ & $I(3,(2)^n,2,1,1)$  & $[(2)^{n+1},2,1,\infty]$ & $[(2)^{n+1},(2,1),3,\infty]$ \\
	2.4.4.1.1&$I(3,(2)^n,1,1,2,1,1)$ & $I(3,(2)^n,2,1,2,1,1)$  & $[(2)^{n+1},2,1,\infty]$ & $[(2)^{n+1},(2,1),3,\infty]$ \\
	2.4.4.1.2&$I(3,(2)^n,1,1,2,1,1)$ & $I(3,(2)^n,2,1,2,1,2)$  & $[(2)^{n+1},2,1,\infty]$ & $[(2)^{n+1},(2,1),3,\infty]$ \\
	2.4.4.1.3&$I(3,(2)^n,1,1,2,1,2)$ & $I(3,(2)^n,2,1,2,1,2)$  & $[(2)^{n+1},2,1,\infty]$ & $[(2)^{n+1},(2,1),3,\infty]$ \\
	2.4.4.1.4&$I(3,(2)^n,1,1,2,1,2)$ & $I(3,(2)^n,2,1,2,1,1)$  & $[(2)^{n+1},2,1,\infty]$ & $[(2)^{n+1},(2,1),3,\infty]$ \\
	2.4.4.2&$I(3,(2)^n,1,1,2,1)$ & $I(3,(2)^n,2,1,2,2)$  & $[(2)^{n+1},2,1,\infty]$ & $[(2)^{n+1},(2,1)^2,20,\infty]$ \\
	2.4.4.3&$I(3,(2)^n,1,1,2,2)$ & $I(3,(2)^n,2,1,2,1)$  & $[(2)^{n+1},2,1,\infty]$ & $[(2)^{n+1},(2,1),3,\infty]$ \\
	2.4.4.4&$I(3,(2)^n,1,1,2,2)$ & $I(3,(2)^n,2,1,2,2)$  & $[(2)^{n+1},2,1,\infty]$ & $[(2)^{n+1},(2,1),3,\infty]$ \\
\end{tabular}
\caption{Case analysis showing the forbidden patterns containing $Z_n=1-X_n-Y_n$ when $n$ is odd.}
\label{tab:my_label_2}
\end{table}

Immediately, from Table~\ref{tab:my_label_1} and Table~\ref{tab:my_label_2}, we see that
\begin{align*}
X_n\in I(3,(2)^n,1)\,, \quad Y_n\in I(3,(2)^n,1) \,,
\end{align*}
is impossible by Case~1. Together, Cases~2.3.1 to 2.3.4 show that
\begin{align}\label{1122}
X_n\in I(3,(2)^n,1,1)\,, \quad Y_n\in I(3,(2)^n,2,2) \,,
\end{align}
is impossible. Next, Cases~2.4.4.1.1 to 2.4.4.1.4 show that it is impossible to have $X_n\in I(3,(2)^n,1,1,2,1)$ and $Y_n\in I(3,(2)^n,2,1,2,1)$. By this and Cases~2.4.4.2 to 2.4.4.4, we get that it is impossible to have $X_n\in I(3,(2)^n,1,1,2)$ and $Y_n\in I(3,(2)^n,2,1,2)$. This, together with Cases~2.4.1 to 2.4.3, shows that
\begin{align}\label{1121}
X_n\in I(3,(2)^n,1,1)\,, \quad Y_n\in I(3,(2)^n,2,1) \,,
\end{align}
is impossible. Next, from Case~2.1, Case~2.2, \eqref{1122}, and \eqref{1121}, we obtain that
\begin{align*}
X_n\in I(3,(2)^n,1)\,, \quad Y_n\in I(3,(2)^n,2) \,,
\end{align*}
is impossible. This analysis of cases shows that any solution of $X_n+Y_n+Z_n=1$ must satisfy
$$
X_n\,, Y_n \in I(3,(2)^{n+1}) \,.
$$
Therefore,
\begin{align}\label{x+y+z=1_2}
x=y=\frac{2-\sqrt{2}}{2} = [3,\overline{2}]\,, \quad z = \sqrt{2}-1 = [\overline{2}]\,,
\end{align}
is the only possible solution of $x+y+z=1$ for this case.

\medskip
\noindent    \textit{Case 4.} $x,y,z \in \B_{2,1}\setminus\B_2$.
\medskip

\noindent We claim that this case is impossible. In fact, $I(3)=[\frac{1}{4},\frac{1}{3})$, so
$x+y+z<1$, which is a contradiction.
Therefore, by \eqref{x+y+z=1_1} and \eqref{x+y+z=1_2}, the theorem is proven.

\subsection{Proof of Theorem~\ref{Main2}}
Observe that $x,y,z \in \B_{2,1}$ satisfy
the equality $x+y+z=1$ if and only if the three equalities
\begin{equation}\label{eq:3eq}
x+y=1-z\,, \quad x + z = 1-y \,, \quad y + z = 1-x \,,
\end{equation}
are also satisfied.  Moreover, by Corollary~\ref{cor:Trivial}, the numbers
$1-z$, $1-y$, and $1-x$ must also be in $\B_{2,1}$.  Next, recall from Theorem~\ref{Main} that
\begin{equation}\label{eq:sol1}
x=2-\sqrt{3}=[3,\overline{1,2}]\,,\quad  y=z=\frac{\sqrt{3}-1}2=[\overline{2,1}]\,,
\end{equation}
and
\begin{equation}\label{eq:sol2}
x=y=\frac{2-\sqrt{2}}2=[3,\overline{2}]\,, \quad z=\sqrt{2}-1=[\overline{2}]\,,
\end{equation}
are the only solutions of the equality $x+y+z=1\,,\ (x,y,z\in\B_{2,1}\,,\ x\leq y \leq z)$.

In \eqref{eq:sol1}, the solution of $x+y+z=1$ happens to have $y=z$, so \eqref{eq:sol1} provides two solutions of our target equality due to \eqref{eq:3eq}.  Specifically, we get that
\[
x=2-\sqrt{3}=[3,\overline{1,2}]\,, \quad y=\frac{\sqrt{3}-1}2=[\overline{2,1}]\,, \quad z=\frac{3-\sqrt{3}}{2}=[1,1,1,\overline{2,1}]\,,
\]
and
\[x=y=\frac{\sqrt{3}-1}2=[\overline{2,1}]\,, \quad z=\sqrt{3}-1=[\overline{1,2}] \,,\]
are both solutions of the equality $x+y=z\,, \ (x,y,z\in \B_{2,1}\,, x\leq y)$.

In \eqref{eq:sol2}, the solution of $x+y+z=1$ happens to have $x=y$, so \eqref{eq:sol2} provides two more solutions of our target equality due to \eqref{eq:3eq}. Specifically, we get that
\[
x=y=\frac{2-\sqrt{2}}2=[3,\overline{2}]\,, \quad z=2-\sqrt{2}=[1,1,\overline{2}]\,,
\]
and
\[x = \frac{2-\sqrt{2}}2=[3,\overline{2}]\,, \quad  y = \sqrt{2}-1=[\overline{2}]\,, \quad z = \frac{\sqrt{2}}{2}=[1,\overline{2}] \,, \]
are also solutions of the equality $x+y=z\,, \ (x,y,z\in \B_{2,1}\,, x\leq y)$.

Finally, we know that these four solutions are the only possibilities because any additional solution to the equality $x+y=z\,, \ (x,y,z\in \B_{2,1}\,, x\leq y)$ would produce an additional solution to $x+y+z=1\,,\ (x\leq y \leq z)$, which is impossible by Theorem~\ref{Main}.

\begin{remark}
Observe that Table~\ref{tab:my_label_2} can be obtained from Table~\ref{tab:my_label_1} by exchanging the left and the right endpoints of the forbidden patterns in the second-last and last columns.
\end{remark}

\subsection{Proof of Theorem \ref{Main3}}
By Theorem~\ref{Main}, we know that there are exactly two solutions of the equality $x+y+z=1$ under the restrictions $x,y,z\in\B_{2,1}$ and $x \leq y\leq z$.  Starting from the second solution
$$
x=y=\frac{2-\sqrt{2}}2=[3,\overline{2}],\ z=\sqrt{2}-1=[\overline{2}] \,,
$$
we can construct further explicit solutions of the equation $x+y+z=1$ under the weaker restrictions $x,y,z \in \B_2^*$ and $x,y,z\in \B_3$.  We do this by making certain insertions into the continued fraction expansions of $x$, $y$, and $z$.

Consider the result of inserting a 2 after the number 3 in the continued fraction expansions of $x$ and $y$, and inserting a 2 foremost in the continued fraction expansion of $z$. This produces new numbers $X$, $Y$, $Z$ that satisfy
$$
X=\cfrac 1{3+\cfrac 1{\cfrac 1{x}-1}}\,, \quad Y=\cfrac 1{3+\cfrac 1{\cfrac 1{y}-1}}\,, \quad
Z=\cfrac 1{2+z} \,.
$$
From these expressions, we obtain the equality
\begin{equation}\label{eq_1_X_Y_Zb}
1-X-Y-Z=\frac{(x-y)^2}{(3-2x)(3-2y)(3-x-y)}\,,
\end{equation}
which implies that $X+Y+Z=1$ because $x=y$. 

Similarly, consider a different type of insertion where we insert $3, 1, 3$ after the number $3$ in the continued fraction expansions of $x$ and $y$, and $1, 1, 2, 1, 1$ after the number $2$ in the continued fraction expansion of $z$.  In this case, we get new numbers $X$, $Y$, and $Z$ that satisfy
$$
X=\cfrac 1{3+\cfrac 1{3+\cfrac 1{1+x}}}\,, \quad Y=\cfrac 1{3+\cfrac 1{3+\cfrac 1{1+y}}}\,, \quad
Z=\cfrac 1{2+\cfrac 1{1+\cfrac 1{1+\cfrac 1{2+\cfrac 1{1+\cfrac 1{\cfrac1{z}-1}}}}}} \,.
$$
This gives us the equality
\begin{equation}\label{eq_1_X_Y_Za}
1-X-Y-Z=\frac{-5(x-y)^2}{(10x+13)(10y+13)(5x+5y+13)}\,,
\end{equation}
which again implies that $X+Y+Z=1$ because $x=y$.

Therefore, there are at least two different types of insertions that result in further solutions of the equality.  Moreover, $S=\{2,11211\}$ is a code, i.e., any word generated by $S$ can be uniquely decomposed into a word over $S$ (see \cite[Chapter~1]{Lothaire:02} for a definition).  From this observation, we immediately get that there are infinitely many solutions of the equality for $x,y,z \in \B_2^*$, since $\B_2^*=\bigcup_{j=1}^\infty\B_{2,j}$.  Furthermore, we obtain uncountably many explicit solutions of the equality for $x,y,z \in \B_3$ by this method.

\begin{remark}
We can also prove Theorem~\ref{Main3} by starting from the first solution of Theorem~\ref{Main} and using the same method with a small modification to the insertions.
\end{remark}

\begin{remark}\label{rem:intuition}
By the formula \eqref{eq_1_X_Y_Zb}, we can reprove Theorem~\ref{Main} by induction.  The right side of \eqref{eq_1_X_Y_Zb} gives the error term after the number $2$ is inserted into the continued fraction expansions of $X_n$, $Y_n$, and $Z_n$ in Tables~\ref{tab:my_label_1} and~\ref{tab:my_label_2}.  In particular, this error term is small enough that the number of cases to consider is always 17.
\end{remark}

\section{Open Problems}\label{sec:open}
In Theorem~\ref{Main}, we prove that there are exactly two solutions of the equality $x+y+z=1$ in $\B_{2,1}$ and that these solutions are in $\Q(\sqrt{2})$ and $\Q(\sqrt{3})$. In $\B_{2,2}$, we obtain at least three additional solutions of the equation $x+y+z=1(x \leq y \leq z)$:
\begin{gather*}
x=y=[3,3,\overline{1,2}]\,, \quad z=[2,1,\overline{1,2}]\,;\\
x=y=[3,1,\overline{1,2}]\,, \quad z=[2,3,\overline{1,2}]\,;\\
x=[3,1,\overline{1,2}]\,, \quad y=[3,3,\overline{1,2}]\,, \quad z=[2,2,2,\overline{2,1}]\,.
\end{gather*}
However, we do not know whether or not the number of solutions is finite in $\B_{2,j}(j \ge 2)$ and whether or not they are in a real quadratic field. In Theorem~\ref{Main3}, we prove that there are infinitely many solutions of the equality $x+y+z=1$ in $\B_2^*$, but we do not know how large the set of solutions is, or its Hausdorff dimension. We know that the number of solutions changes from finite to infinite as we go from $\B_{2,1}$ to $\B_2^*$, but what happens in between?

The proof of Theorem \ref{Main3} relied on the two
key identities \eqref{eq_1_X_Y_Zb} and \eqref{eq_1_X_Y_Za}.
In fact, many similar identities can be found.
For example,
\begin{align*}
1-[2,1,3,\tfrac{1}{x}-1]-[2,1,3,\tfrac{1}{y}-1]-[3,1,1,\tfrac{1}{1-x-y}]&=\frac{4(x-y)^2}{(8x-11)(8y-11)(11-4x-4y)}\,,\\
1-[3,1,1,\tfrac{1}{x}]-[3,1,1,\tfrac{1}{y}]-[2,3,1,\tfrac{1}{1-x-y}-1]&=-\frac{2(x-y)^2}{(4x+7)(4y+7)(2x+2y+7)}\,,\\
1-[3,3,1,\tfrac{1}{x}-2]-[3,3,1,\tfrac{1}{y}-2]-[2,1,1,1,1-x-y]&=\frac{8(x-y)^2}{(16x-13)(16y-13)(
	13-8x-8y)}\,.
	\end{align*}
	From such identities whose right side is divisible by $x-y$,
	we can construct
	further explicit solutions of $x+y+z=1$ in $\B_2^*$ or $\B_3$.
	It may be an interesting problem to characterize the set of such identities.
	Such transformations on $x$ form a semi-group of
	integral M\"obius transformations, but can we describe its generators?
	Are they finite?
	
	Moreover, our method produces many isosceles triangles, but we know very little
	about scalene triangles.
	Below, we point out a
	sporadic infinite family of solutions to $x+y+z=1$ of this type:
	$$
	x=[3,(2)^{\ell},1,\overline{1,2}]\,, \quad
	y=[3,(2)^{\ell},3,\overline{1,2}]\,, \quad
	z=[(2)^{4+2\ell},\overline{1,2}] \qquad \ell=0,1,\dots \,,
	$$
	which is shown by induction using the two lucky equalities:
	\begin{align*}
1-[3,\tfrac{1}{x}-1]-[3,\tfrac{1}{y}-1]-[2,2,\tfrac{1}{1-x-y}]&=
\frac{2 (x+y-3) (2 x y-2 x-2 y+1)}{(2 x-3) (2 y-3) (7-2 x-2 y)}\,,\\
2[3,\tfrac{1}{x}-1][3,\tfrac{1}{y}-1]
-2[3,\tfrac{1}{x}-1]-2[3,\tfrac{1}{y}-1]+1
&=-
\frac{2 xy -2 x-2 y+1}{(2 x-3) (2 y-3)}\,.
\end{align*}

Lastly, regarding the diffraction from the Delone sets obtained in Section~\ref{sec:specific}, an interesting open problem is to determine how the continued fraction expansions of the badly approximable numbers $x$, $y$, $z$ are related to the autocorrelation measures of the associated Delone sets; see \cite[Chapter~9]{Baake-Grimm:13} for an overview of the theory of diffraction from point sets.  A different construction of Delone sets via badly approximable angles and Fermat Spirals is discussed in \cite{Adiceam-Tsokanos:22, Akiyama:20,  Marklof:20, Yamagishi-Sushida-Sadoc:21}; the diffraction from these examples may also be an interesting avenue of further study.

\appendix

\section{Proof of Compactness of \texorpdfstring{$W_{r,R}(\RR^2)$}{Wr,R(R2)}}\label{sec:compactness}

For completeness, here we give a proof that $W_{r,R}(\RR^2)$ is compact in the Chabauty--Fell topology.
Note that to obtain
the desired compactness of $W_{r,R}(\RR^2)$, it is necessary to use open balls for uniform
discreteness and closed balls for relative denseness
in the definition of an $(r,R)$-Delone set.
\begin{lem}
\emph{
	The set
	\[
	\cW_{r,R}(\RR^2) = \{Y \in 2^{\RR^2} \,:\, Y \ \text{is an} \ (r,R)\text{-Delone set in} \ \RR^2\} \,,
	\]
	is compact in the Chabauty--Fell topology.
}
\end{lem}
\begin{proof}
It follows from the early work of Fell \cite{Fell:62} that the space $2^X$ is compact in the Chabauty--Fell topology for every topological space $X$. Let $X = \RR^2$.  Since $\cW_{r,R}(X) \subseteq 2^X$, it suffices to show that $\cW_{r,R}(X)$ is closed.  To this end, let $\{\Lambda_n\}_{n\geq 1}$ be a sequence in $\cW_{r,R}(X)$ converging to some $\Lambda \in 2^X$.  Our goal is to prove that $\Lambda$ is both $r$-uniformly discrete and $R$-relatively dense. 

\emph{\textbf{Proof of $r$-uniform discreteness.}}  Let $x\in X$ be arbitrary.  Suppose by contradiction that there are two points $y,z\in \Lambda\cap B(x,r)$ with $y\neq z$.  Consider any
$$
0 < \eps < \min(\textstyle{\frac{1}{r+\|x\|}},r-\|x-y\|,r-\|x-z\|,\textstyle{\frac{1}{2}}\|y-z\|) \,.
$$
By convergence in the Chabauty--Fell topology, there exists an $N\geq 1$ such that
$$
\Lambda \cap B(0,\textstyle{\frac{1}{\eps}}) \subseteq \Lambda_n + B(0,\eps) \quad \forall n\geq N\,.
$$
Now, since $\eps < \frac{1}{r + \|x\|}$, we have that $B(x,r)\subseteq B(0,\textstyle{\frac{1}{\eps}})$, so
$$
y, z \in \Lambda \cap B(x,r)  \subseteq \Lambda \cap B(0,\textstyle{\frac{1}{\eps}}) \subseteq \Lambda_n + B(0,\eps) \quad \forall n \geq N\,.
$$
In particular, there exist points $y_N$ and $z_N$ in $\Lambda_N$ such that $\|y-y_N\|<\eps$ and $\|z-z_N\| < \eps$.  By our choice of $\eps$, we have $\|y-z\| > 2\eps$.  From this, we see that
$$
\|y_N - z_N\| \geq \|y - z\| - \|y-y_N\| - \|z - z_N\| > 2\eps - \eps - \eps = 0 \,,
$$
so $y_N \neq z_N$.  Furthermore, we have $r-\|x-y\| > \eps$ and $r-\|x-z\| > \eps$, so $y_N$ and $z_N$ are both in $B(x,r)$.  Indeed, we have
$$
\|x - y_N\| \leq \|x-y\| + \|y-y_N\| < r-\eps + \eps = r\,,
$$
and a similar inequality shows that $\|x-z_N\| < r$.  This contradicts $r$-uniform discreteness of $\Lambda_N$.  Therefore, $\Lambda\cap B(x,r)$ has at most one element for every $x\in \RR^2$, i.e.\ $\Lambda$ is $r$-uniformly discrete.

\emph{\textbf{Proof of $R$-relative denseness.}}  Let $x\in X$ be arbitrary.  We aim to show that $\Lambda \cap \overline{B(x,R)} \neq \emptyset$.  By assumption, we have that $\Lambda_n$ is $R$-relatively dense for every $n$, so there exists a sequence of points $y_n$ in $X$ such that $y_n \in \Lambda_n \cap \overline{B(x,R)}$ for all $n$.  Since $\overline{B(x,R)}$ is compact in $X$ (in the Euclidean topology), there exists a subsequence $y_{n_i}$ of $y_n$ and some $y\in\overline{B(x,R)}$ such that $\|y_{n_i} - y\| \rightarrow 0$. Next, observe that for all $n\in\N$ with $n > R+\|x\|+1$, we have $\overline{B(x,R)} \subseteq B(0,n)$.  From this and convergence in the Chabauty--Fell topology, there exists an $N>R+\|x\|+1$ such that
$$
y_{n} \in \Lambda_n \cap \overline{B(x,R)} \subseteq \Lambda_n \cap B(0,n) \subseteq \Lambda + B(0,\textstyle{\frac{1}{n}}) \quad \forall n\geq N\,.
$$
Thus, there is a sequence $z_n$ in $\Lambda$ such that $\|y_n - z_n\| < \frac{1}{n}$ for all $n\geq N$.  In particular, we have $z_{n_i}\rightarrow y$.  Now, since $\Lambda$ is a closed set in the Euclidean topology, we must have $y\in \Lambda$.  Moreover, $y$ is also in $\overline{B(x,R)}$, so we have shown that there is some $y\in \Lambda \cap \overline{B(x,R)} \neq \emptyset$, as desired.  This completes the proof.
\end{proof}

\section*{Acknowledgements}
S.A. is supported by JSPS grants 20K03528, 24K06662 and RIMS in Kyoto University.
E.R.K. is supported by the Natural Sciences and Engineering Research Council of Canada (NSERC) via grant 2024-0485 and the Canadian Graduate Scholarship - Doctoral.
Y-L.X. is supported by the China Scholarship Council, China (No.\,202306770085).  We would also like to thank Nicolae Strungaru for his insights on the construction of Delone sets, and Noel Murasko for his assistance in producing the graphics. We are grateful to the anonymous referee for giving us comments to increase the readability of the original manuscript. The introduction is greatly improved by this suggestion.

\end{document}